%% file: Confidence_Bands_Histogram-R1-arxiv.tex
\documentclass[12pt,twoside,english]{article}
\usepackage[T1]{fontenc}
\usepackage[latin9]{inputenc}
\usepackage[a4paper]{geometry}
\usepackage{babel}
\usepackage{amsmath}
\usepackage{amsthm}
\usepackage{amssymb}
\usepackage{booktabs}
\usepackage{multirow}
\usepackage[authoryear]{natbib}
\usepackage[unicode=true,pdfusetitle,
 bookmarks=true,bookmarksnumbered=false,bookmarksopen=false,
 breaklinks=false,pdfborder={0 0 1},backref=false,colorlinks=false]
 {hyperref}

\makeatletter
\theoremstyle{plain}
\newtheorem{thm}{\protect\theoremname}[section]
\theoremstyle{plain}
\newtheorem{lem}[thm]{\protect\lemmaname}
\theoremstyle{definition}
\newtheorem{ass}[thm]{Assumption}

\newtheorem{example}[thm]{Example}
\theoremstyle{plain}

\theoremstyle{remark}

\theoremstyle{thm}
\newtheorem{cor}[thm]{\protect\corollaryname}

\@ifundefined{date}{}{\date{}}
\usepackage{amsmath}

\DeclareMathOperator{\var}{Var}
\DeclareMathOperator{\cov}{Cov}
\DeclareMathOperator{\E}{\mathbb{E}}
\DeclareMathOperator{\diam}{diam}
\DeclareMathOperator{\vol}{vol}
\usepackage{algorithm, algorithmicx}
\usepackage{algpseudocode}
\usepackage{graphicx}
\usepackage{tikz}
\usepackage{setspace}
\algnewcommand\algorithmicinput{\textbf{Input:}}
\algnewcommand\Input{\item[\algorithmicinput]}
\algnewcommand\algorithmicoutput{\textbf{Output:}}
\algnewcommand\Output{\item[\algorithmicoutput]}
\allowdisplaybreaks
\usepackage{bbm}

\renewcommand{\hat}{\widehat}
\renewcommand{\tilde}{\widetilde}

\makeatother

\providecommand{\definitionname}{Definition}
\providecommand{\lemmaname}{Lemma}
\providecommand{\propositionname}{Proposition}
\providecommand{\remarkname}{Remark}
\providecommand{\theoremname}{Theorem}
\providecommand{\corollaryname}{Corollary}

\begin{document}
\global\long\def\var{\mathrm{Var}}%
\global\long\def\cov{\mathrm{Cov}}%
\global\long\def\E{\mathbb{E}}%
\global\long\def\P{\mathbb{P}}%
\global\long\def\QED{\hfill\square}%

\title{Asymptotic confidence bands for the histogram regression
estimator}
\author{Natalie Neumeyer$^1$, Jan Rabe$^1$ and Mathias Trabs$^2$}
\date{Universit\"at Hamburg$^1$ and Karlsruhe Institute of Technology$^2$}
\maketitle
\begin{abstract}
Asymptotic uniform confidence bands are constructed for a multivariate nonparametric regression model with heteroscedastic noise, employing histogram estimators under flexible partition conditions. The construction is especially applicable to unsmooth  regression functions of H\"older regularity less than one. While the radius of the confidence bands could be approximated via the Gumbel distribution, our construction does not depend on an extreme value distribution, but instead can be explicitly calculated for the chosen partition.
\end{abstract}


\section{Introduction}

We consider histogram estimators for the regression function $m$ in a nonparametric regression model $Y=m(X)+\varepsilon$ with multivariate covariates $X$ and heteroscedastic noise. Our aim is to derive asymptotic histogram based confidence bands. 

There is a substantial amount of literature on confidence bands based
on different nonparametric regression estimators. \citet{Johnston1982} has
constructed confidence bands for the Nadaraya-Watson kernel estimator.
\citet{Haerdle1989} has shown asymptotic uniform confidence bands for
a wider class of regression estimators, the M-smoothers. Typically,
confidence bands rely on an undersmoothing of the estimator to
reduce the bias relative to the stochastic error. An alternative is
a direct bias correction, which is used, for instance, by \citet{Eubank1993},
who have considered a deterministic, uniform design for local constant
regression estimation, and by \citet{Xia1998} considering a random design under dependence and local linear
estimation. There are also bootstrap confidence bands for nonparametric
regression, see for instance \citet{Hall1993}, \citet{Neumann1998},
\citet{Claeskens2003} or \citet{Hall2013}. 
Based on spline estimators regression confidence bands were constructed by \citet{Wang-Yang} in the random covariate case, and by \citet{Krivobokova-etal} in the deterministic design case. \citet{Cai-etal} considered confidence bands for the mean and the variance function for fixed design. These confidence bands are based on Nadaraya-Watson-type estimators, where for the variance case the residuals are estimated based on spline-estimators.
\citet{Sabbah2014} has shown
confidence bands for quantile estimators which is an alternative to
a regression model, and \citet{Birke2010} have proved confidence bands
in an inverse regression model.
\\
The aforementioned articles only considered univariate regression. For multidimensional covariates often semiparametric regression models are considered due to the curse of dimensionality. For instance, confidence bands for the link function in a single-index model were developed by  \cite{Gu-Yang}, for additive regression functions by \cite{Haerdle-etal}, and for additive semi-parametric regression functions by \cite{cattaneo2024nonlinear}. In contrast we focus on multivariate nonparametric regression models without additional structural assumptions.
In the multivariate case \citet{Konakov1984} have analyzed the asymptotic
distribution of the maxi\-mal deviation for the Nadaraya-Watson estimate
in a random design setting. \citet{Proksch2016} has proved confidence
bands in a multivariate regression model with fixed, deterministic
design for a general class of estimators that includes, for example,
local polynomial estimators. \citet{Chao2017} have covered confidence
bands for multivariate quantile and expectile regression with mean regression as a special case.
\\
Due to their easy and efficient implementation and clear interpretability histogram regression estimators (also called regressogram) enjoy a high popularity and widespread application, see e.g.\ \cite{Gyoerfi2002,wasserman2006}. Also in the development of advanced statistical methodologies, histogram estimators continue to serve as a fundamental building block, e.g.\ for regression trees \cite{Breiman1984} or in the context of privacy, see \cite{berrett2021strongly}.
In a density estimation setting the study of confidence bands goes back to \citet{Smirnov1950}. The regression histogram estimators are the simplest case of partitioning-based estimators whose asymptotic properties are studied in depth by \cite{Cattaneo-etal}. The latter article also includes confidence bands with a focus on general (bias corrected) partitioning-based estimators. For the case of multivariate histogram estimators the applicability of their construction is however limited. We want to close this gap by presenting an easy to implement confidence band result in a multivariate regression model for histogram estimators under very mild conditions on the partition. This especially allows us to incorporate a-priori knowledge on the regression function into the construction of the partition to improve the finite sample performance.
\\
In an undersmoothing regime we construct an honest asymptotic uniform confidence band in the class of $\alpha$-H\"older regular regression functions for $\alpha\in(0,1]$. Note that all above mentioned articles on confidence bands require the regression function to be at least continuously differentiable while we will allow for regression functions which are less than Lipschitz regular. The diameter of our confidence band is naturally determined by the locally constant variance of the histogram estimator together with quantiles of the distribution of the maximum $\max_{j=1,\dots,\Delta}|Z_j|$ of independent standard normal random variables $(Z_j)_{j\ge1}$ where $\Delta$ is given by the number of cells of the partition of the histogram. Whereas our confidence intervals do not need to be based on an extreme value distribution, our critical values can be related to the quantiles of the Gumbel distribution. We also incorporate the estimation of the unknown variance and a possible unknown distribution of the covariates.
\\
We obtain a confidence band result for the simple case of a H\"older continuous regression function for every covariate dimension, without a continuity assumption on the covariate density. In contrast, for smooth nonparametric regression estimators it is typically assumed that for increasing covariate dimension the number of existing derivatives of the regression function and the covariate density has to increase, see, for instance, \cite{Chao2017}. Instead we only exploit finite higher moments of the error distribution, an assumption which is especially satisfied for finite exponential moments. 

 A central part of the proof for the histogram based confidence band is the application of a result by \citet{Chernozhukov2014}
regarding the Gaussian approximation for the suprema of empirical
processes.
Using this direct approximation of the supremum, it is not necessary
to approximate a whole empirical process uniformly.
This method has also been applied by \citet{Patschkowski2019} to construct adaptive confidence bands for
probability densities.
It is a key difference to previous results in the literature like those by \citet{Johnston1982},
\citet{Claeskens2003} or \citet{Chao2017}, and  the proof
method might be extended to different estimators.

The main result is presented and discussed in Section~\ref{sec:CB histo}. The proof strategy is outlined in Section~\ref{sec:Proof-strategy}. In Section \ref{sec:simus} we show a numerical illustration. Some conclusions are given in Section \ref{sec:conclusion}. Detailed proofs are postponed to Section~\ref{sec:Proofs} and in Appendix \ref{sec:extension} extensions are discussed.

\section{Histogram estimator and confidence bands}\label{sec:CB histo}

On some probability space $(\Omega,\mathcal{A},\mathbb{P})$ we observe
the training sample $(X_{i},Y_{i})_{i=1}^{n}$ consisting
of $n$ observations which are independent copies of the random variable
$(X,Y)$ with $X\in\mathcal X$, $Y\in\mathbb{R}$ and where $\mathcal X\subseteq\mathbb R^p$ is some bounded feature space. The goal is
to estimate the regression function $m\colon\mathcal X\to\mathbb{R}$ with
\begin{equation*}
Y=m(X)+\varepsilon 
\end{equation*}
for an observation error $\varepsilon$ with $\mathbb{E}[\varepsilon|X]=0$ and variance function $\sigma^2(x)=\var (\varepsilon|X=x)$.
We further assume that the
density $f_{X}$ of the covariate or feature vector $X$ and the variance function are bounded from below and above, i.e.
\begin{eqnarray}\label{cX}
0<c_{X}\leq f_{X}(\cdot)\leq C_{X},\\
0<c_{\sigma^2}\leq \sigma^2(\cdot)\leq C_{\sigma^2}\label{csigma}
\end{eqnarray}
for some constants $c_X,C_X,c_{\sigma^2},C_{\sigma^2}>0$.

The histogram estimator relies on a local constant approximation of $m$. For a given partition of the feature space $\mathcal X$ the regression function is estimated by the cell-wise average of the response variables in each cell of the partition. We will impose the following regularity assumption on the partition:
\begin{ass}\label{ass:part}
  Let $\bigcup_{k=1}^\Delta A_k=\mathcal X$ be a partition of $\mathcal X$ of size $\Delta=\Delta_n\in\mathbb N$ into disjoint subsets $A_k\subset\mathcal X,k=1,\dots,\Delta$, satisfying for some constants $c_p,C_p>0$ that
  \begin{equation}\label{ass:cell}c_p\Delta^{-1}\le\vol(A_k)\le C_p\Delta^{-1}\text{ and }\diam(A_k)\le C_p \Delta^{-1/p}\text{ for all }k=1,\dots,\Delta
  \end{equation}
  with volume $\vol(A_k)$ and diameter $\diam(A_k):=\sup_{x_{1},x_{2}\in A_k}\Vert x_{1}-x_{2}\Vert$ of $A_k$.
\end{ass}
On the one hand, the assumption guarantees that the overall volume $\vol(\mathcal X)$ is roughly equally distributed over the cells of the partition. The condition on diameter ensures that the cells of the partition do not degenerate in any direction, but instead spread similarly in all directions. On the other hand, Assumption~\ref{ass:part} allows for a flexible choice of the partition such that, e.g.\ prior knowledge of the regression function could be incorporated in the construction to improve the finite sample performance. In Section~\ref{sec:simus} we will see a numerical example in this direction.

For $x\in\mathcal X$ let $A_{\Delta}(x)$ be the subset $A_k$ that contains $x$.
The histogram estimator is then defined by
\begin{equation}\label{hat-m}
\hat{m} (x)=\sum_{j=1}^{n}Y_{j}\frac{\mathbb{I}\{X_{j}\in A_{\Delta}(x)\}}{\sum_{i=1}^{n}\mathbb{I}\{X_{i}\in A_{\Delta}(x)\}} 
\end{equation}
where $\mathbb{I}\{\cdot\}$ denotes the indicator function. 
If there is no observation in $A_{\Delta}(x)$ we define the estimator
as zero on this cell. The simplicity of histogram estimators comes with some computational advantages. First, they have lower computational costs than kernel estimators when many evaluations are needed ($O((m\wedge \Delta)n)$ instead of $O(mn)$ for $m$ point evaluations). Second, they serve as summary statistics to compress data and reduce storage space from $n$ observations to only $\Delta$ cell-wise means.

In the following theorem we present a confidence band result based on the histogram estimator. 
Let $\mathcal{H}^\alpha(C_{H})$ denote the set of all bounded $\alpha$-H\"older functions $m\colon\mathcal X\to\mathbb{R}$ with H\"older regularity $\alpha\in(0,1]$ and both H\"older constant and supremum norm bounded by $C_H>0$.

\begin{thm}
\label{thm:CB histo} Let the partition satisfy Assumption~\ref{ass:part}. Let $\alpha\in(0,1],C_H>0$ and assume that
\begin{equation}
n\Delta^{-1-2\alpha/p}(\log n)^2\to0.\label{eq:assum histo approx}
\end{equation}
Assume that $\sup_{x\in \mathcal X}\mathbb{E}[|\varepsilon|^3|X=x]<\infty$ and $\mathbb{E}\left[\vert\varepsilon\vert^{\nu}\right] <\infty$ for some $4\leq \nu\in\mathbb{N}$ such that
\begin{align}
\frac{\Delta(\log n)^{5}}{n^{1-2/\nu}} & \to0.\label{eq:assum histo GP approx}
\end{align}
For independent standard normally distributed
random variables $(Z_{j})_{j\in\mathbb{N}}$ and some $\beta\in(0,1)$ let $c_{\Delta}(\beta)>0$ be given by
\begin{equation}\label{eq:c}
\mathbb{P}\left(\max_{j=1,\ldots,\Delta}\vert Z_{j}\vert\leq c_{\Delta}(\beta)\right)=1-\beta.
\end{equation}
For 
\begin{equation}\label{eq:tau}\tau_\Delta(x)
=\frac{\mathbb{P} (X\in A_\Delta(x))^2}{\E[\sigma^2(X)\mathbb{I}_{A_\Delta(x)}(X)]}
\end{equation}
the confidence band
\[
\mathcal{C}_{n}(x)=\left[\hat{m} (x)-c_{\Delta}(\beta)(\tau_\Delta(x)n)^{-1/2},\hat{m} (x)+c_{\Delta}(\beta)(\tau_\Delta(x)n)^{-1/2}\right]
\]
satisfies
\[
\liminf_{n\to\infty}\inf_{m\in\mathcal{H}^\alpha(C_{H})}\mathbb{P}\left(m(x)\in\mathcal{C}_{n}(x),\,\forall x\in\mathcal X\right)\geq1-\beta.
\]
\end{thm}

The proof strategy is explained in Section \ref{sec:Proof-strategy}, and the proofs are given in Section~\ref{sec:Proofs}. While the theorem assumes the covariate distribution and variance function to be known, we will consider a replacement of $\tau_\Delta(x)$ by an estimator below. Note that by Gaussian concentration, we have
\begin{equation}\label{Rate-S}
\max_{j=1,\ldots,\Delta}\vert Z_{j}\vert= \mathcal{O}_{\mathbb{P}}(\sqrt{2\log \Delta}) = \mathcal{O}_{\mathbb{P}}(\sqrt{\log n})    
\end{equation}
due to $\log \Delta\lesssim\log n$ which is implied by assumption (\ref{eq:assum histo GP approx}). Therefore, $c_\Delta(\beta)\lesssim \sqrt{\log n} $.
Moreover, the quantile $c_\Delta(\beta)$ can be explicitly represented in terms of the quantile function of the standard normal distribution since
$$ 
  1-\beta=\mathbb{P}\left(\max_{j=1,\ldots,\Delta}\vert Z_{j}\vert\leq c_{\Delta}(\beta)\right)
  =\mathbb{P}\left(\vert Z_{1}\vert\leq c_{\Delta}(\beta)\right)^\Delta
  =(2\Phi(c_\Delta(\beta))-1)^\Delta.
$$
with cdf $\Phi$ of $\mathcal N(0,1)$. Consequently, $c_\Delta(\beta)=\Phi^{-1}(\frac{1}{2}(1+(1-\beta)^{1/\Delta}))$ is easy to compute.

Define
\[
p_\Delta(x)=\mathbb{P}(X\in A_{\Delta}(x))=\int_{A_\Delta(x)}f_X(t)\,dt.
\]
By (\ref{cX}) and (\ref{csigma}) we obtain 
\begin{equation}\label{eq:bounds p_x_0}
c_{X}\vol(A_\Delta(x))\leq p_\Delta(x)\leq C_{X}\vol(A_\Delta(x))
\end{equation}
and
\begin{equation}\label{eq:bounds tau_x}
\frac{c_{X}}{C_{\sigma^2}}\vol(A_\Delta(x))\leq \tau_\Delta(x)\leq \frac{C_{X}}{c_{\sigma^2}}\vol(A_\Delta(x)).
\end{equation}
Due to Assumption~\ref{ass:part} we have $\vol(A_\Delta(x))=\mathcal O(\Delta^{-1})$ uniformly in $x\in\mathcal X$ and thus the diameter of the confidence band $\mathcal C_n$ is of order $\sqrt{(\log n)\Delta/n}$ uniformly on $\mathcal X$.

\begin{example}\label{example}
  Let $\mathcal X=[0,1]^p$ and consider a partition based on an equidistant grid such that all cells $(A_k)_{k=1,\dots,\Delta}$ are hypercubes of the same size. Denoting the edge length of the cubes by $\delta>0$, we have $\Delta=\delta^{-p}$, $\diam(A_k)=\sqrt p\delta=\sqrt p\Delta^{-1/p}$ and $\vol(A_k)=\delta^p=\Delta^{-1}$. 
    In the homoscedastic case with constant $\sigma^2(\cdot)\equiv \sigma^2$ we obtain $\tau_\Delta(x)=p_\Delta(x)/\sigma^2$. For uniformly distributed covariates on $[0,1]^p$ we have $p_\Delta(x)=\Delta^{-1}$ and in that case the confidence band simplifies to
\[
\mathcal{C}_{n}(x)=\Big[\hat{m} (x)-c_{\Delta}(\beta)\sigma\sqrt{\frac{\Delta}{n}},\hat{m} (x)+c_{\Delta}(\beta)\sigma\sqrt{\frac{\Delta}{n}}\Big].
\]\end{example}

The assumptions (\ref{eq:assum histo approx}) and (\ref{eq:assum histo GP approx}) in Theorem~\ref{thm:CB histo} 
are up to logarithms that
$$
n/\Delta^{1+2\alpha/p}  \to0 \text{ and } n/(\Delta n^{2/\nu})  \to\infty,
$$
where $2/\nu$ is arbitrary small if all moments of $\varepsilon$ exist.
Since the approximation error of the histogram estimator is of order $\Delta^{-\alpha/p}$ while the stochastic error term is of order $(\Delta/n)^{1/2}$, the first assumption is an undersmoothing condition. The second condition
implies that the average number of observations in all cells goes to infinity such that the estimator is consistent. As $p$ increases, the assumptions become
more restrictive which is the common curse of dimensionality problem for nonparametric function estimation. However, the assumptions 
are not contradictory, and for every dimension $p$ no bias reduction is needed.  

\smallskip

The special case of confidence bands based on a histogram with a cubic partition is also covered by \citet[Theorem~6.1 in combination with Remark~6.2]{Cattaneo-etal}. For general partitions in dimensions $p\ge 2$ Theorem~6.2 in \cite{Cattaneo-etal} requires more regular regression functions compared to Theorem~\ref{thm:CB histo}. More precisely, they need $\Delta^3/n\to 0$ (in our notation) instead of \eqref{eq:assum histo GP approx}. An optimal choice $\Delta=n^{p/(2\alpha+p)}$ (up to logarithms, see Corollary~\ref{cor-rate}) satisfies $\Delta^3/n\to 0$ only for $\alpha>p$. \cite{Cattaneo-etal} also assume differentiability of the regression function. Theorem~\ref{thm:CB histo} does not require that the regression function's and covariate density's regularities have to grow with the dimension, which is a common restriction in the literature on nonparametric regression with multivariate covariates, see  \cite{Proksch2016} and \cite{Chao2017} for confidence bands, and  \cite{Dette-Derbort} and \cite{Neumeyer-VanKeilegom}, among others, for different estimation and testing problems.

\smallskip

Typically, confidence bands regression functions rely on an extreme value distribution that is independent of $\Delta$. Based on Theorem~\ref{thm:CB histo} and the Gumbel extreme value distribution, we can derive an analogous result. Defining
\begin{eqnarray*}
    a_\Delta&=&\sqrt{2\log\Delta}\qquad\text{and}\\
    b_\Delta&=&\sqrt{2\log\Delta}-\frac{\log(\log\Delta)+\log(4\pi)-\log(4)}{\sqrt{8\log\Delta}},
\end{eqnarray*}
the critical value $c_\Delta(\beta)$ can be approximated by 
\begin{equation}\label{eq:cTilde}
\tilde c_\Delta(\beta)=b_\Delta-\frac{\log(-\log(1-\beta))}{a_\Delta}.
\end{equation}
Let the confidence interval $\tilde{\mathcal{C}}_{n}(x)$ be defined as $\mathcal{C}_{n}(x)$, but replacing $c_\Delta(\beta)$ by $\tilde c_\Delta(\beta)$. The proof of the corollary is given in Section~\ref{sec:Proofs}. 
\begin{cor}\label{limit-distribution} Under the assumptions of Theorem \ref{thm:CB histo} it holds that
\[
\liminf_{n\to\infty}\inf_{m\in\mathcal{H}^\alpha(C_{H})}\mathbb{P}\left(m(x)\in\tilde{\mathcal{C}}_{n}(x),\,\forall x\in\mathcal X\right)= 1-\beta.
\]
\end{cor}

\smallskip

\noindent
In contrast to the corollary Theorem~\ref{thm:CB histo} provides bands with explicit dependence on $\Delta$, and the
accuracy of the bands does not depend on the rate of convergence to
that limit distribution. Indeed, the approximation accuracy of $\tilde c_\Delta(\beta)$ increases only slowly in the number of cells $\Delta$ and consistently overestimates $c_\Delta(\beta)$ as shown in Figure~\ref{fig:quantiles}. Applying an approximation by the maximum of Gaussian random variables instead of a limit distribution is also suggested by \cite{Kreiss-etal}
for confidence bands for the spectral density in time series analysis.

\begin{figure}\centering
  \include{img/quantiles2}
  \caption{Absolute error $\tilde c_\Delta(\beta)- c_\Delta(\beta)$ and relative error $(\tilde c_\Delta(\beta)- c_\Delta(\beta))/c_\Delta(\beta)$ for $\beta=0.05$ in dependence of $\Delta$.}\label{fig:quantiles}
\end{figure}

From Theorem \ref{thm:CB histo} we obtain a uniform rate of convergence of the histogram estimator. Then we obtain the following corollary whose proof is given in Section~\ref{sec:Proofs}. 

\begin{cor}\label{cor-rate}
     Let the partition satisfy Assumption~\ref{ass:part}. Assume $\sup_{x\in \mathcal X}\mathbb{E}[|\varepsilon|^3|X=x]<\infty$ and $\mathbb{E}\left[\vert\varepsilon\vert^{\nu}\right] <\infty$ for some $4\leq \nu\in\mathbb{N}$ such that $(\log n)^{5}\Delta/n^{1-2/\nu} \to0$.
   Then we have 
\[
\sup_{x\in\mathcal X}|\hat m(x)-m(x)|=\mathcal{O}_{\mathbb{P}}\left(\Delta^{-\alpha/p}+\sqrt{(\log n)\Delta/n}\right).
\]
For $\nu>2+p/\alpha$ and $\Delta=\lfloor (n/\log n)^{p/(2\alpha+p)}\rfloor$ we obtain the uniform rate of convergence $\mathcal O_{\mathbb P}((n/\log n)^{-\alpha/(2\alpha+p)})$.
  \end{cor}
This convergence rate coincides the minimax-optimal rate for the uniform loss, cf. \cite{Tsybakov2009}. The proof is given in Section~\ref{sec:Proofs}. 

\smallskip

We will now consider the case where the covariate distribution and the variance function of the noise are unknown. To this end, let $\hat\tau_\Delta(x)$ be an estimator for $\tau_\Delta(x)$ from \eqref{eq:tau} and define 
 the confidence band as 
\[
\hat{\mathcal{C}}_{n}(x)=\left[\hat{m} (x)-c_{\Delta}(\beta)(\hat \tau_\Delta(x)n)^{-1/2},\hat{m} (x)+c_{\Delta}(\beta)(\hat \tau_\Delta(x)n)^{-1/2}\right].
\]

\begin{cor}\label{cor-estimated-px}
    Under the assumptions of Theorem \ref{thm:CB histo} and if
    $$\sup_{x\in\mathcal X}\left|\frac{\hat\tau_\Delta(x)}{\tau_\Delta(x)}-1\right|=o_{\mathbb{P}}((\log n)^{-3/2}),$$
    we have
\[
\liminf_{n\to\infty}\inf_{m\in\mathcal{H}^\alpha(C_{H})}\mathbb{P}\left(m(x)\in \hat{\mathcal{C}}_{n}(x),\,\forall x\in\mathcal X\right)\geq1-\beta.
\]
  \end{cor}

The proof is again postponed to Section~\ref{sec:Proofs}.  To construct the estimator $\hat\tau_\Delta(x)$, we first consider the homoscedastic case where $\tau_\Delta(x)=p_\Delta(x)/\sigma^2$ can be estimated by $\hat \tau_\Delta(x)=\hat p_\Delta(x)/\hat\sigma^2$ with 
\begin{equation}
\label{hatpx}
\hat p_\Delta(x)=\frac1n \sum_{i=1}^n \mathbb{I}\{X_{i}\in A_{\Delta}(x)\}.
\end{equation}
Note that $\hat p_\Delta(x)$ is the density histogram estimator for the density function $f_X$ of the covariate distribution. This estimator has the desired rate which we show in Lemma~\ref{lem:pHat}. The constant variance can be estimated using residuals based on the histogram regression estimator, i.e.
$$\hat\sigma^2=\frac1n \sum_{i=1}^n \big(Y_i-\hat m(X_i)\big)^2.$$
Other variance estimators under homoscedasticity and variance function estimators under heteroscedasticity are considered by \cite{Shen2020}, among others.  Nevertheless, one can also apply a histogram estimator for the variance function $\sigma^2(x)=\E[Y^2|X=x]-m(x)^2$ defined by
\begin{eqnarray*}
    \hat\sigma_\Delta^2(x)&=&\frac{\sum_{i=1}^n Y_i^2\mathbb{I}\{X_i\in A_\Delta(x)\}}{\sum_{i=1}^n \mathbb{I}\{X_i\in A_\Delta(x)\}}-\hat m(x)^2\\
    &=&\frac{\sum_{i=1}^n (Y_i-\hat m(x))^2\mathbb{I}\{X_i\in A_\Delta(x)\}}{\sum_{i=1}^n \mathbb{I}\{X_i\in A_\Delta(x)\}}.
\end{eqnarray*}
Then the uniform rate of convergence $\mathcal{O}_{\mathbb{P}}(\sqrt{(\log n)\Delta/n})$ holds as in Corollary \ref{cor-rate} under the same assumptions for observations $(X,Y^2)$, in particular under moment conditions and H\"older continuity of $\sigma^2(\cdot)$. 
The histogram estimator $\hat\sigma_\Delta^2(\cdot)$ is constant on each subset $A_\Delta(x)$ and thus 
\begin{equation}\label{s_x}
    s_\Delta(x) = 
    \E\big[\sigma^2(X)\mathbb I_{A_\Delta(x)}(X)\big]=\int_{A_\Delta(x)}\sigma^2(t)f_X(t)\,dt
\end{equation}
is approximated by $\int_{A_\Delta(x)} \hat\sigma_\Delta^2(x)f_X(t)\,dt=\hat\sigma_\Delta^2(x)p_\Delta(x)$. Hence, we can estimate $\tau_\Delta(x)=p_\Delta(x)^2/s_\Delta(x)$ by
$$\hat\tau_\Delta(x)=\frac{\hat p_\Delta(x)}{\hat\sigma_\Delta^2(x)}$$
and obtain the rate assumed in Corollary \ref{cor-estimated-px}.

\section{Proof strategy\label{sec:Proof-strategy}}


In this section we outline the proof. The estimation error can be decomposed as
\begin{equation}
\hat{m} (x)-m(x)=\hat{m}^{(m)}(x)-m(x)+\hat{m}^{(\varepsilon)}(x)\label{eq:error decomp histo}
\end{equation}
with approximation error $\hat{m}^{(m)}(x)-m(x)$ and stochastic error $\hat{m}^{(\varepsilon)}(x)$, where
\begin{eqnarray*}
    \hat{m} ^{(m)}(x)&=&\sum_{j=1}^{n}m(X_{j})\frac{\mathbb{I}\{X_{j}\in A_{\Delta}(x)\}}{\sum_{i=1}^{n}\mathbb{I}\{X_{i}\in A_{\Delta}(x)\}}\qquad\text{and}\\
    \hat{m}^{(\varepsilon)}(x)&=&\sum_{j=1}^{n}\varepsilon_{j}\frac{\mathbb{I}\{X_{j}\in A_{\Delta}(x)\}}{\sum_{i=1}^{n}\mathbb{I}\{X_{i}\in A_{\Delta}(x)\}}.
\end{eqnarray*}

First, we discuss a uniform bound on the approximation error.
Note that $\hat{m}^{(m)}(x)=\mathbb{E}\left[\hat{m}(x)\mid(X_{j})_{j=1}^{n}\right]$ corresponds to the regression estimate in a setting without observation noise. Therefore,
the approximation error is indeed caused by the ability of the piecewise
constant estimator to approximate $m$. Moreover, note that the expectation
of the approximation error is the bias of the estimator.
%
Suppose that there
exists at least one observation in $A_{\Delta}(x)$. Then we use that $m$
is $\alpha$-H\"older continuous for $\alpha\in(0,1]$ with H\"older constant
$C_{H}$ to obtain uniformly in $x\in\mathcal{X}$ that
\begin{align}
\vert\hat{m}^{(m)}(x)-m(x)\vert & =\Big\vert\sum_{j=1}^{n}(m(X_{j})-m(x))\frac{\mathbb{I}\{X_{j}\in A_{\Delta}(x)\}}{\sum_{i=1}^{n}\mathbb{I}\{X_{i}\in A_{\Delta}(x)\}}\Big\vert\nonumber \\
 & \leq C_{H}\sum_{j=1}^{n}\Vert X_{j}-x\Vert^{\alpha}\frac{\mathbb{I}\{X_{j}\in A_{\Delta}(x)\}}{\sum_{i=1}^{n}\mathbb{I}\{X_{i}\in A_{\Delta}(x)\}}\nonumber \\
 & \leq C_{H}\diam(A_{\Delta}(x))^{\alpha}=C_HC_p^\alpha\Delta^{-\alpha/p}, \label{eq:approx error bound histo}
\end{align}
where we used $\diam(A_{\Delta}(x))\le C_p\Delta^{-1/p}$ by Assumption~\ref{ass:part}. In particular, the approximation error profits from a finer partition.

To handle the stochastic error, we can replace the denominator in $\hat m^{(\varepsilon)}(x)$ by its expectation, such that
\begin{equation}\label{tilde-m}
\tilde{m}^{(\varepsilon)}(x)=\frac{1}{n}\sum_{j=1}^{n}\varepsilon_{j}\frac{\mathbb{I}\{X_{j}\in A_{\Delta}(x)\}}{p_\Delta(x)}
\end{equation}
becomes the asymptotically leading term of the stochastic error $\hat{m}^{(\varepsilon)}(x)$.
Handling the remainder term $\hat m^{(\varepsilon)}(x)-\tilde m^{(\varepsilon)}(x)$ uniformly in $x$ relies on a moment bound for the binomial
distribution that can be found in  Section~\ref{sec:binom}.

The idea for the leading term is that $\tilde{m}^{(\varepsilon)}$
is close to a Gaussian process in $x$ if $n$ is large enough.
We write $\tau_\Delta(x)=p_\Delta(x)^2/s_\Delta(x)$, see \eqref{eq:tau} and \eqref{s_x}.
Note that $p_\Delta(x)$, $\tau_\Delta(x)$ and $s_\Delta(x)$  only depend on $x$ via the cell $A_\Delta(x)$, so they take only finitely many values for $x\in\mathcal X$.   We now calculate the covariance
\begin{align*}
    \cov(\tilde m^{(\varepsilon)}(x_1), \tilde m^{(\varepsilon)}(x_2))
    &= \frac1n \mathbb{E}\Big[ \mathbb{E}\left[ \varepsilon^2|X\right]\frac{\mathbb{I}\{X\in A_{\Delta}(x_1)\}\mathbb{I}\{X\in A_{\Delta}(x_2)\}}{p_\Delta (x_1)p_\Delta (x_2)}\Big]\\
    &= \mathbb{I}\{A_{\Delta}(x_1)=A_{\Delta}(x_2)\}\frac1n \frac{s_\Delta (x_1)}{p_\Delta (x_1)^2}\\
    &= \mathbb{I}\{A_{\Delta}(x_1)=A_{\Delta}(x_2)\}(n\tau_\Delta(x_1))^{-1}.
\end{align*}
To obtain a process with unit variance, we consider
\begin{equation}\label{eq:gaussPr}
  (n\tau_\Delta(x))^{1/2}\tilde m^{(\varepsilon)}(x)= \mathbb{G}_{n}f_{x,\Delta}
\end{equation}
for the empirical process $\mathbb{G}_{n}f=\sqrt{n}( P_nf- Pf)$ based on independent $(X_1,\varepsilon_1),\dots,(X_n,\varepsilon_n)$ with distribution $P$ and empirical distribution $P_n$, indexed by $f\in \mathcal{F}_\Delta$ with
the finite function class
\begin{equation}
\mathcal{F}_{\Delta}:=\big\{f_{x,\Delta}\colon\mathcal X\times\mathbb{R}\to\mathbb{R}\mid f_{x,\Delta}(t,\epsilon)=\epsilon\,s_\Delta(x)^{-1/2}\mathbb{I}\{t\in A_{\Delta}(x)\}, x\in\mathcal X\big\}\label{eq:F delta}.
\end{equation}
We apply the result  by \citet{Chernozhukov2014} formulated in Theorem \ref{thm:sup-approx Cherno} in the Appendix~\ref{appendix}
to approximate
$\sup_{f\in\mathcal{F}_{\Delta}}\vert\mathbb{G}_{n}f\vert$ 
by a sequence of random variables $\mathbf{S}_{\Delta}$ with 
$$\mathbf{S}_{\Delta}\overset{d}{=}\sup_{f\in\mathcal{F}_{\Delta}}\vert B_{\Delta}(f)\vert.$$
Therein, $B_{\Delta}$ is a sequence of centered Gaussian processes
indexed by $\mathcal{F}_{\Delta}$ whose covariance function is given by
\begin{align}\nonumber
\cov(B_{\Delta}(f_{x_{1},\Delta}),B_{\Delta}(f_{x_{2},\Delta})) & =\mathbb{E}\left[f_{x_{1},\Delta}(X_{1},\varepsilon_{1})f_{x_{2},\Delta}(X_{1},\varepsilon_{1})\right]\\
 \label{eq:covariance}
 & =\mathbb{I}\{A_{\Delta}(x_{1})=A_{\Delta}(x_{2})\}.
\end{align}
In particular, this implies $\var(B_{\Delta}(f_{x,\Delta}))=1$ and that the process
on the finite function class is in fact a family of independent Gaussian
random variables. The application of Theorem \ref{thm:sup-approx Cherno} thus leads to the following theorem, which is proved in Section
\ref{sec:Proofs}.

\begin{thm}
\label{thm:sup m tilde approx} Grant Assumption~\ref{ass:part}  and let $\log \Delta=\mathcal O(\log n)$. 
Assume that  $\mathbb{E}[\vert\varepsilon_{1}\vert^{\nu}]<\infty$ for some   $\nu\in[4,\infty)$ and $\sup_{x\in\mathcal X}\mathbb{E}[|\varepsilon|^3|X=x]<\infty$. Then there exists a sequence of random variables $$\mathbf{S}_{\Delta}\overset{d}{=}\max_{j=1,\ldots,\Delta}\vert Z_{j}\vert$$
with independent standard normally distributed $(Z_{j})_{j\in\mathbb N}$
such that
\[
\bigg\vert\sqrt{n}\sup_{x\in\mathcal X}\vert \tau_\Delta(x)^{1/2}\tilde{m}^{(\varepsilon)}(x)\vert-\mathbf{S}_{\Delta}\bigg\vert=\mathcal{O}_{\mathbb{P}}\left(\frac{(\log n)^{3/2}\Delta^{1/2}}{n^{1/2-1/\nu}}+\frac{(\log n)^{5/4}\Delta^{1/4}}{n^{1/4}}+\frac{(\log n)\Delta^{1/6}}{n^{1/6}}\right).
\]
\end{thm}

As mentioned above this direct approximation of the supremum avoids to approximate the whole empirical process uniformly by a Gaussian process. As a result we need weaker assumptions compared to more classical confidence bands derivations.

From the covariance calculation above and the assumptions on the covariate density and variance function it follows that the variance of the stochastic error is given by
\[
\var\left(\tilde{m}^{(\varepsilon)}(x)\right)=\frac{1}{n\tau_\Delta(x)}=\mathcal O\Big(\frac\Delta{n}\Big),
\]
i.e.\ the stochastic error profits from a partition into fewer cells as larger cells lead to more $\varepsilon_{j}$
being averaged. Compared to the bound of the approximation error in \eqref{eq:approx error bound histo}, we recover the typical bias-variance trade-off in nonparametric statistics. The assumptions on $\Delta$ in Theorem~\ref{thm:CB histo} lead to the commonly used undersmoothing. That means $\Delta$ has to be chosen such that the stochastic error is slightly larger than the approximation error, allowing the confidence bands to be built based on the distribution of the stochastic error.

\section{Numerical illustration}\label{sec:simus}
To illustrate our theoretical performance in a simulation example, we consider the following model: Let $p=2$, $X\sim U([0,1]^2)$ and the regression function by given by
\[
  m(x)=3\sqrt{\big|0.36-\|x\|^2\big|}-(x_1+x_2),\qquad x=(x_1,x_2)\in[0,1]^2.
\]
Note that $m$ is H\"older regular of order $\alpha\le1/2$ and that $m$ is close the a radial function.
We choose homogeneous Gaussian noise with known noise level $\sigma>0$ and fix the sample size $n=800$. We use histogram estimators based on two different partitions. First, we use a histogram regression estimator $\hat m_\text{square}$ based on a partition into $\Delta=h^{-2}$ squares of edge length $h=1/6$. Second, we consider an estimator $\hat m_\text{circ}$ which relies on a partition based on circular segments. More precisely, we decompose the unit square into 15 circles and split the middle ones into two to four segments. Both partitions consist of $\Delta=36$ cells. For comparison we also calculate a Nadaraya-Watson estimator $\hat m_\text{nw}$ based on a Gaussian kernel. The bandwidth is chosen either as $h/\sqrt{12}$, such that the standard deviation of the kernel corresponds to that of a uniform distribution on a $h\times h$ square, or via the automatic bandwidth choice implemented in the \texttt{R} package \texttt{np}. The true regression function, both histogram estimators as well as the Nadaraya-Watson estimator are illustrated in Figure~\ref{fig:estimates}.

\begin{figure}
    \centering
    \includegraphics[width=0.49\linewidth]{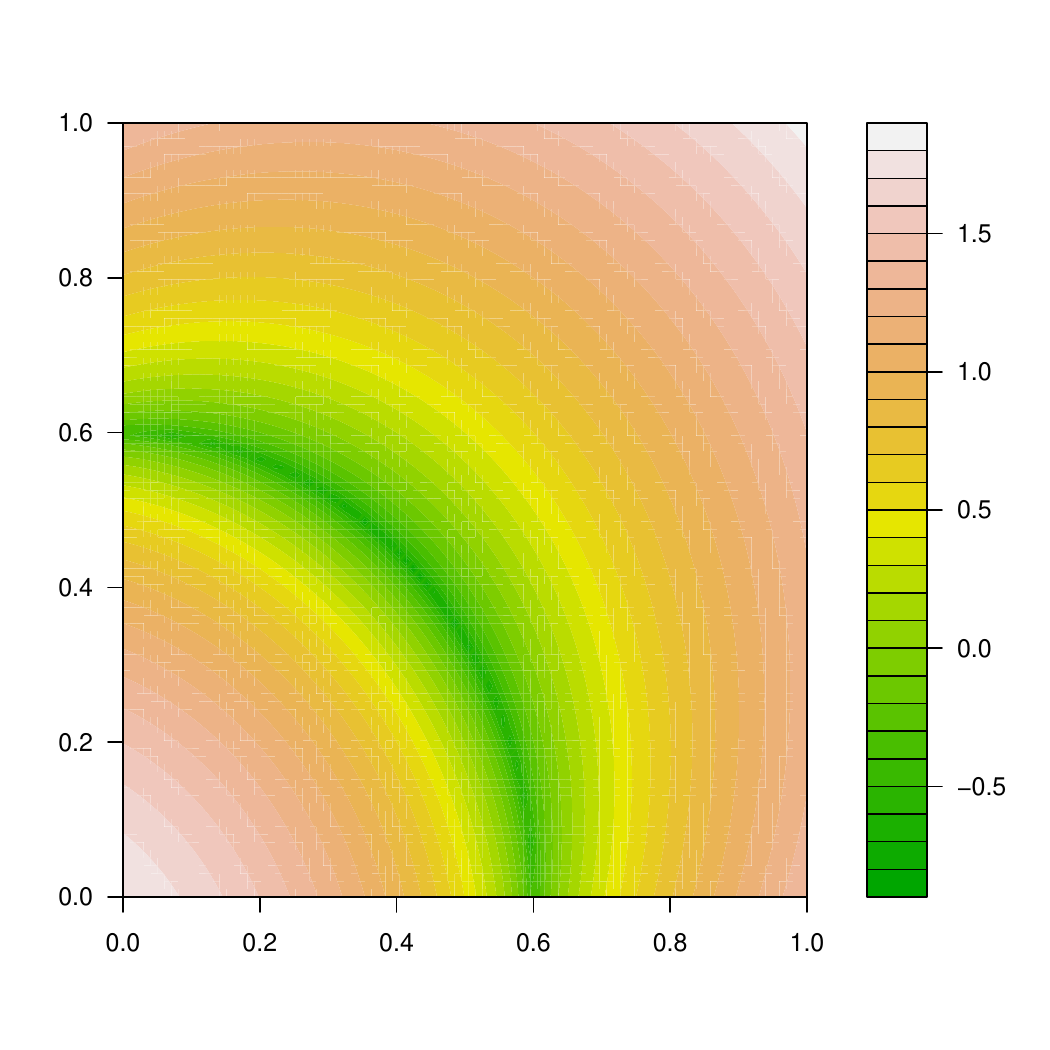}
    \includegraphics[width=0.49\linewidth]{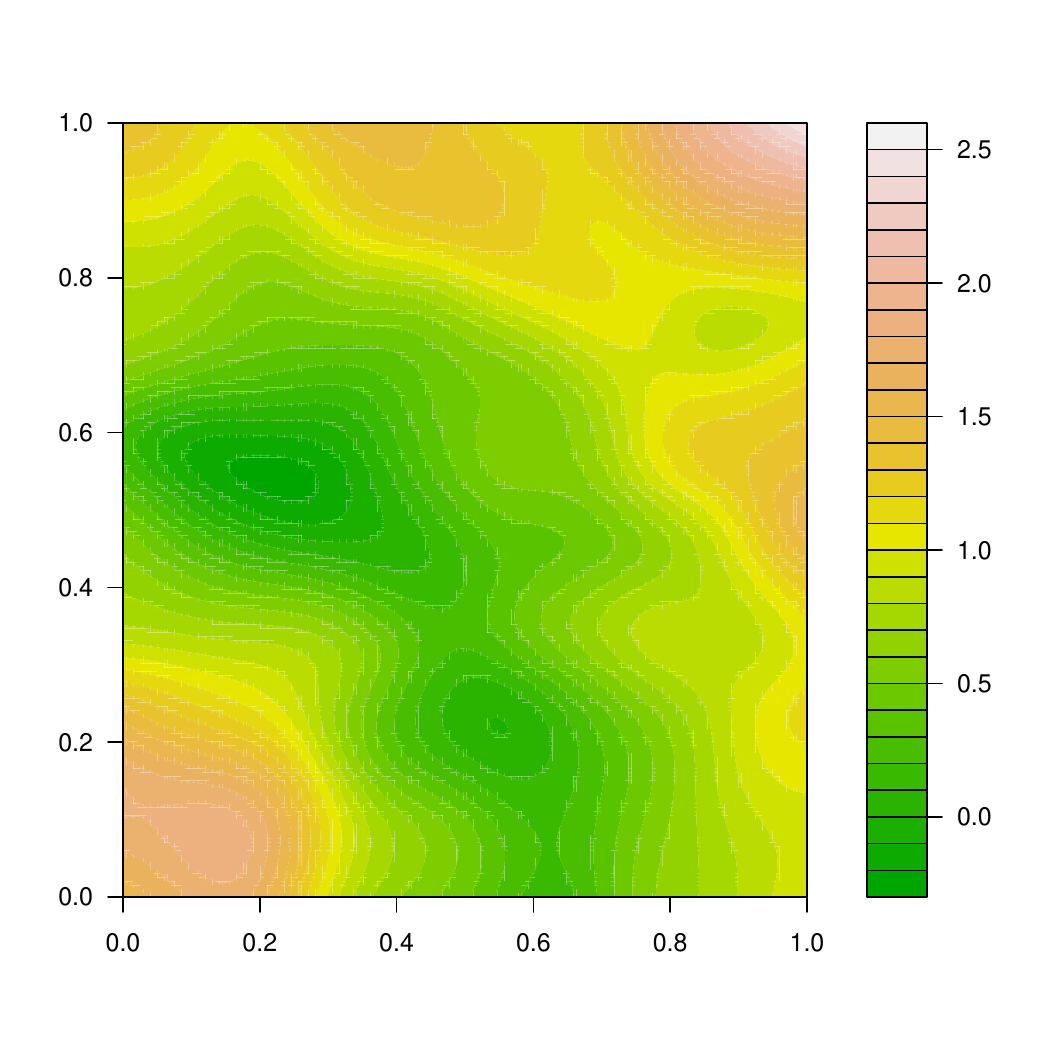}
    \includegraphics[width=0.49\linewidth]{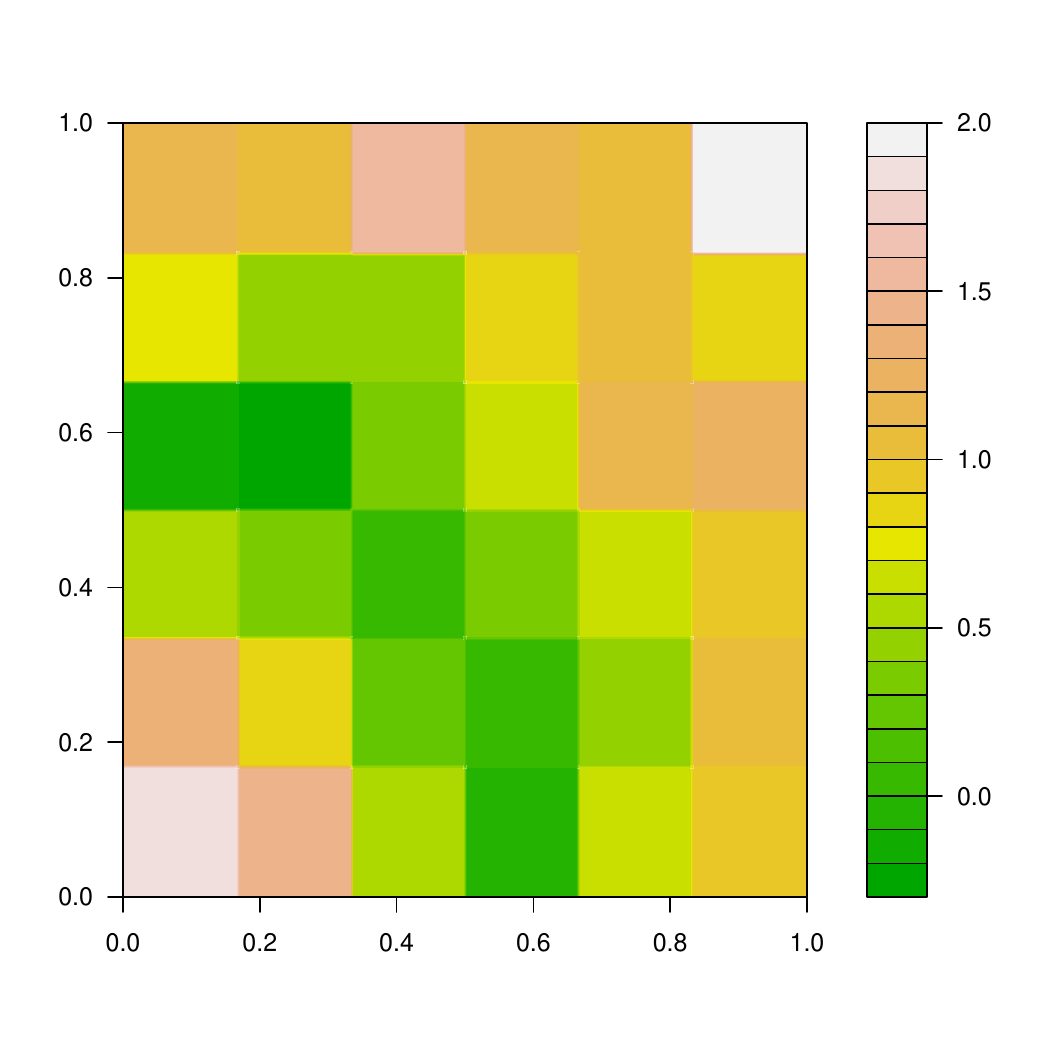}
    \includegraphics[width=0.49\linewidth]{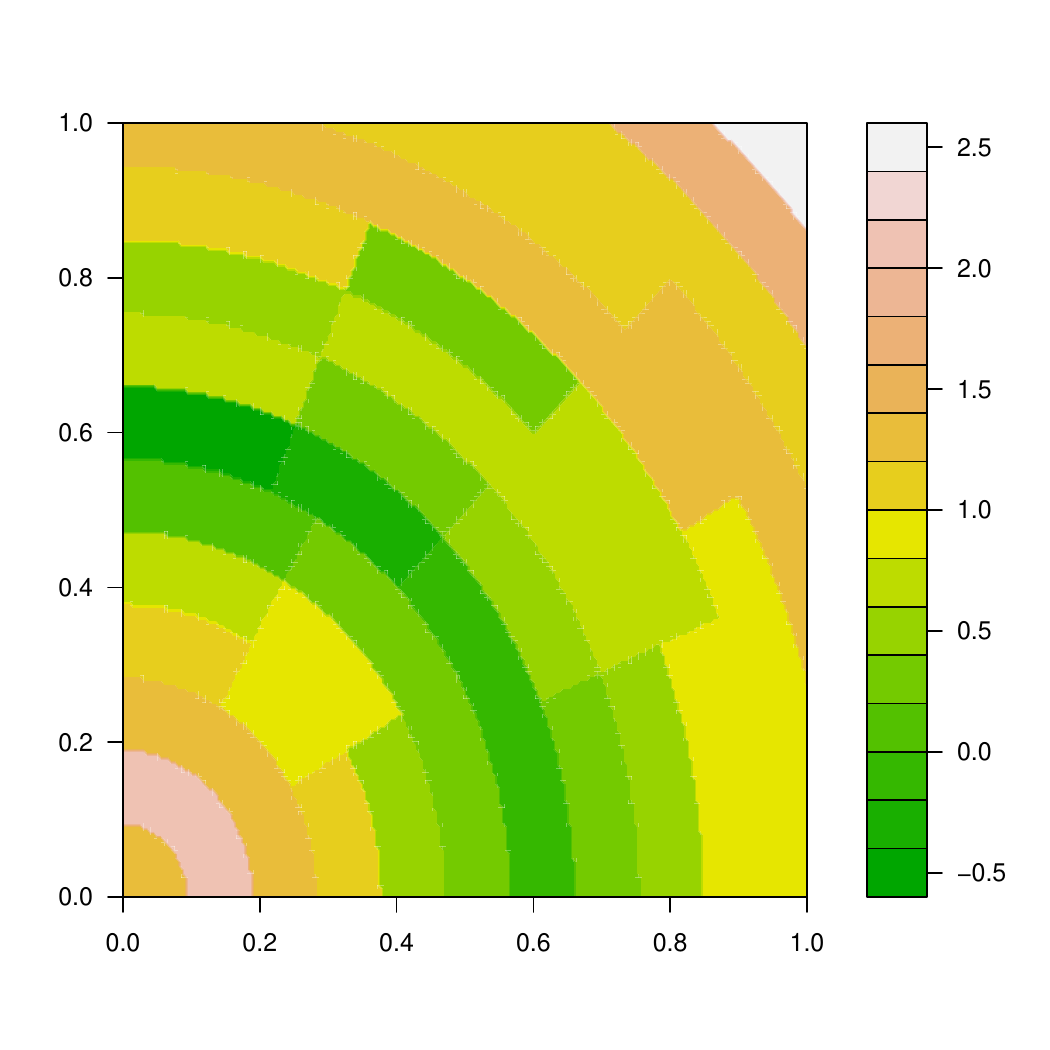}
    \caption{Regression function (top left) together with realizations of the Nadaraya-Watson estimator with automatic bandwidth choice (top right) and the two histogram estimators (bottom row) with two different partitions based on $n=800$ observations with noise level $\sigma=1$.}
    \label{fig:estimates}
\end{figure}

\begin{table}[t]
\centering
\begin{tabular}{ccccccc}
\toprule
\multicolumn{2}{c}{$\sigma$}  & 1 & 2 & 3 & 4 & 5 \\
\midrule
\multicolumn{2}{c}{$R(\hat m_\text{circ})$}& 0.9010 & 1.4501 & 2.0383 & 2.6376 & 3.2394 \\
\multicolumn{2}{c}{$R(\hat m_\text{square})$} & 1.1895 & 1.5800 & 2.0947 & 2.6338 & 3.1805 \\
\multicolumn{2}{c}{$R(\hat m_\text{nw})$ (fixed $h$)} & 1.0076 & 1.8232 & 2.7005 & 3.5900 & 4.4815 \\
\multicolumn{2}{c}{$R(\hat m_\text{nw})$ (adaptive $h$)} & 0.9446 & 1.2055 & 1.4417 & 1.6066 & 1.7530 \\
\midrule
\multirow{2}{*}{CP for $\hat m_\text{circ}$} &$c_\Delta(0.95)$& 0.99 & 1.00 & 1.00 & 1.00 & 1.00 \\
&$\tilde c_\Delta(0.95)$ & 1.00 & 1.00 & 1.00 & 1.00 & 1.00 \\
\multirow{2}{*}{CP for $\hat m_\text{square}$} &$c_\Delta(0.95)$ & 0.98 & 0.99 & 0.99 & 1.00 & 1.00 \\
&$\tilde c_\Delta(0.95)$ & 0.98 & 1.00 & 1.00 & 1.00 & 1.00 \\
\midrule
\multirow{2}{*}{UCP for $\hat m_\text{circ}$} &$c_\Delta(0.95)$& 0.05 & 0.61 & 0.78 & 0.85 & 0.88 \\
&$\tilde c_\Delta(0.95)$ & 0.15 & 0.71 & 0.85 & 0.93 & 0.95 \\
\multirow{2}{*}{UCP for $\hat m_\text{square}$} &$c_\Delta(0.95)$ & 0.00 & 0.15 & 0.49 & 0.58 & 0.61 \\
&$\tilde c_\Delta(0.95)$ & 0.00 & 0.26 & 0.54 & 0.63 & 0.70 \\
\bottomrule
\end{tabular}
\caption{Mean uniform estimation errors $R(\hat m)$ for histogram estimators based on circular and quadratic partitions and mean uniform estimation error of a Nadaraya-Watson estimator; average pointwise coverage probabilities (CP) and uniform coverage probabilities (UCP) of the confidence bands with nominal level $0.95$ based on both histogram estimators with exact and asymptotic critical value, all for different noise levels $\sigma$.}\label{tab:numResults}
\end{table}

We compare the performance of the three estimators in terms of the mean uniform estimation error $R(\hat m)=\E[\sup_{x\in[0,1]^2}|\hat m(x)-m(x)|]$ based on a grid on $[0,1]^2$ with mesh size 0.005 and a Monte Carlo approximation of the expectation with 100 Monte Carlo iterations. Moreover, we calculate the uniform empirical coverage probabilities of the confidence bands with nominal confidence level $0.95$ based on the two histogram estimators and using the critical values $c_\Delta$ and $\tilde c_\Delta$ from \eqref{eq:c} and \eqref{eq:cTilde}, respectively. As an additional feedback on the coverage properties of the confidence bands, we also report the average pointwise coverage probabilities over the spatial grid. Since the radius of the confidence bands based on $\tilde c_\Delta$ are by the factor $\frac{\tilde c_{16}(0.95)}{c_{16}(0.95)}\approx 1.0453$ larger than those based on $c_\Delta$, the coverage probabilities of the latter are throughout smaller. To see the dependence of the signal-to-noise ratio we vary the noise level $\sigma\in\{1,2,\dots,5\}$. All results are summarized in Table~\ref{tab:numResults}. We see that for small noise levels the performance of the histogram estimators is comparable to that of the Nadaraya-Watson estimator, while the coverage of the confidence bands is poor. By increasing the noise level, we enter an under-smoothing regime such that the estimation error becomes dominated by the stochastic error and the coverage considerably improves. Naturally, the estimation error grows with $\sigma$. The increasing deviation to the (adaptive) Nadaraya-Watson estimator emphasizes that the partitions are coarse in these cases.

The partition of $\hat m_{\text{circ}}$ is better adapted to the geometry of the regression function. As a consequence, the coverage probabilities for $\hat m_{\text{circ}}$ are considerably better than for $\hat m_{\text{square}}$ in all cases. Moreover, for noise levels $\sigma=1,2,3$ the estimation error of  $\hat m_{\text{circ}}$ is smaller, while for $\sigma=4,5$ the estimation error of $\hat m_{\text{square}}$ is only slightly better.

\section{Conclusion}\label{sec:conclusion}
Histogram regression estimators are used a lot in applications and serve as a building block for more advanced statistical methods such as random forests. Aiming for uniform statistical inference, we construct histogram based confidence bands which allow for a general choice of partitions. This flexibility can improve the finite sample behaviour of the histogram estimator depending on the data structure. 
\\
Our confidence band is applicable for possibly irregular nonparametric regression functions with H\"older-regularity only of order $\alpha\in (0,1]$. The authors are not aware of any uniform confidence band for non-differentiable regression functions in the literature. 
These mild model assumptions are possible due to a proof structure, which deviates from more classical proof methods and which could also be transferred to other methods and models.
For instance, the current preprint \cite{neumeyer2025asymptotic} demonstrates that a similar proof strategy can be used to derive confidence bands for centered purely random forests. 
Some ideas to obtain confidence bands for other nonparametric regression estimators are discussed in Appendix \ref{sec:extension} for different choices of cells including the case of a Nadaraya-Watson estimator.

\section{Proofs\label{sec:Proofs}}

Subsequently we gather all the proofs starting with our main result Theorem~\ref{thm:CB histo}. We will write $A\lesssim B$ equivalently for $A=\mathcal O(B)$. Moreover, we set $K=\sup_{x\in\mathcal{X}}\mathbb{E}[|\varepsilon|^3|X=x]<\infty$, and use the notation $\Vert \cdot\Vert_{\infty}$ for the supremum norm on $\mathcal{X}$.

\subsection{Proof of Theorem \ref{thm:CB histo} and its corollaries}

\begin{proof}[Proof of Theorem \ref{thm:CB histo}]
 We use the decomposition of the estimator and error from (\ref{eq:error decomp histo}) and the definition of $\tilde m^{(\varepsilon)}$ in (\ref{tilde-m}). For the sequence of random variables $\mathbf{S}_{\Delta}$
as defined in Theorem~\ref{thm:sup m tilde approx} we obtain by the reverse triangle inequality that
\begin{align}\nonumber
\bigg\vert\sqrt{n}  \sup_{x\in\mathcal X}&\tau_\Delta(x)^{1/2}\vert\hat{m}(x)-m(x)\vert-\mathbf{S}_{\Delta}\bigg\vert
\leq  T_1 + T_2 + T_3,\qquad \text{with}\\
 T_1& =
 \sqrt{n}\sup_{x\in\mathcal X}\tau_\Delta(x)^{1/2}\big\vert\hat{m}^{(m)}(x)-m(x)\big\vert, \label{approxterm}\\
 T_2&=\bigg\vert\sqrt{n}\sup_{x\in\mathcal X}\vert \tau_\Delta(x)^{1/2}\tilde{m}^{(\varepsilon)}(x)\vert-\mathbf{S}_{\Delta}\bigg\vert,\label{approx-leadingterm}\\
 T_3&=\sqrt{n}\sup_{x\in\mathcal X}\tau_\Delta(x)^{1/2}\big\vert\tilde{m}^{(\varepsilon)}(x)-\hat{m}^{(\varepsilon)}(x)\vert.\label{msigmaremainder}
\end{align}
Subsequently, we will verify that all three terms are $o_{\mathbb{P}}((\log n)^{-1})$.

For $T_1$ from \eqref{approxterm} note that the approximation error satisfies
\begin{align}
\hat{m}^{(m)}(x)-m(x) & =\sum_{j=1}^{n}(m(X_{j})-m(x))\frac{\mathbb{I}\{X_{j}\in A_{\Delta}(x)\}}{\sum_{i=1}^{n}\mathbb{I}\{X_{i}\in A_{\Delta}(x)\}}\nonumber \\
 & \qquad+m(x)(\mathbb{I}\{\exists j\in[n]:X_{j}\in A_{\Delta}(x)\}-1).\label{eq:prf histo approx err decomp}
\end{align}
Using \eqref{eq:approx error bound histo}, we get
\begin{align}
\mathbb{E}\left[\sup_{x\in\mathcal X}\Big\vert\sum_{j=1}^{n}(m(X_{j})-m(x))\frac{\mathbb{I}\{X_{j}\in A_{\Delta}(x)\}}{\sum_{i=1}^{n}\mathbb{I}\{X_{i}\in A_{\Delta}(x)\}}\Big\vert\right] & \leq C_{H}\mathbb{E}\left[\sup_{x\in\mathcal X}\diam(A_{\Delta}(x))^{\alpha}\right]\nonumber \\
 & =C_{H}C_p^\alpha\Delta^{-\alpha/p}.\label{eq:prf histo approx err part 1 bound}
\end{align}
Moreover, we obtain by \eqref{eq:bounds p_x_0}
\begin{align*}
\mathbb{E} & \left[\sup_{x\in\mathcal X}\vert m(x)(\mathbb{I}\{\exists j\in[n]:X_{j}\in A_{\Delta}(x)\}-1)\vert\right]\\
 & \leq\Vert m\Vert_{\infty}\mathbb{E}\left[\sup_{x\in\mathcal X}\mathbb{I}\{\nexists j\in[n]:X_{j}\in A_{\Delta}(x)\}\right]\\
 & \leq\Vert m\Vert_{\infty}\Delta\sup_{x\in\mathcal X}\mathbb{E}\big[\mathbb{I}\{\nexists j\in[n]:X_{j}\in A_{\Delta}(x)\}\big]\\
 & =\Vert m\Vert_{\infty}\Delta\sup_{x\in\mathcal X}(1-p_\Delta(x))^{n}
 \leq\Vert m\Vert_{\infty}\Delta(1-c_{X}{c_p\Delta^{-1}})^{n}.
\end{align*}
To handle (\ref{approxterm}), note that $n^{1/2}\sup_{x}|\tau_\Delta(x)|^{1/2}\le (C_XC_p/c_{\sigma^2})^{1/2}(n/\Delta)^{1/2}$ by (\ref{eq:bounds tau_x}). Together with (\ref{eq:prf histo approx err decomp}) and (\ref{eq:prf histo approx err part 1 bound}) we obtain for all $\kappa>0$
\begin{align*}
\mathbb P\left(T_1\ge \frac{\kappa}{\log n}\right)&=\mathbb{P}  \left(\Vert m-\hat{m}^{(m)}\Vert_{\infty}\geq\kappa\frac{(c_{\sigma^2}/(C_X)C_p)^{1/2}\Delta^{1/2}}{n^{1/2}\log n}\right)\\
 & \lesssim\mathbb{E}\big[\Vert m-\hat{m}^{(m)}\Vert_{\infty}\big]\kappa^{-1}(n/\Delta)^{1/2}\log n\\
 & \lesssim{(\Delta^{-2\alpha/p}n/\Delta)^{1/2}\log n+\Delta(1-c_{X}c_p\Delta^{-1})^{n}(n/\Delta)^{1/2}\log n}.
\end{align*}
The first term is $o(1)$ due to (\ref{eq:assum histo approx}) and
for the second term we have
\[
\Delta(1-c_{X}c_p\Delta^{-1})^{n}(n/\Delta)^{1/2}\log n\lesssim\exp(-c_{X}c_pn/\Delta)\Delta(n/\Delta)^{1/2}\log n\to0
\]
because $\Delta\ge n^{1/(2\alpha/p+1)}$ by \eqref{eq:assum histo approx} and $n/\Delta\ge n^{2\nu}$ by \eqref{eq:assum histo GP approx}. Therefore, $T_1=o_{\mathbb{P}}((\log n)^{-1})$.

For the term $T_2$ from (\ref{approx-leadingterm})  Theorem \ref{thm:sup m tilde approx} yields
\begin{align}
T_2&=\bigg\vert\sqrt{n}\sup_{x\in\mathcal X}\vert \tau_\Delta(x)^{1/2}\tilde{m}^{(\varepsilon)}(x)\vert-\mathbf{S}_{\Delta}\bigg\vert\notag\\
&=\mathcal{O}_{\mathbb{P}}\left(\frac{(\log n)^{3/2}\Delta^{1/2}}{n^{1/2-1/\nu}}+\frac{(\log n)^{5/4}\Delta^{1/4}}{n^{1/4}}+\frac{(\log n)\Delta^{1/6}}{n^{1/6}}\right)\label{eq:prf histo GP approx err}
\end{align}
because $\nu\geq4$. 
To see that the right-hand side is indeed $o_{\mathbb{P}}((\log n)^{-1})$, note that assumption (\ref{eq:assum histo GP approx})
implies
\[
\frac{(\log n)^{5/2}\Delta^{1/2}}{n^{1/2-1/\nu}}=\Big(\frac{(\log n)^{5}\Delta}{n}n^{2/\nu}\Big)^{1/2}\to0.
\]
The same argument applied to the other terms in (\ref{eq:prf histo GP approx err})
yields $T_2=o_{\mathbb{P}}((\log n)^{-1})$. 

Next we consider the remainder term $T_3$ from (\ref{msigmaremainder}). By the definition of $\tilde m^{(\varepsilon)}$ in  (\ref{tilde-m}) and $\hat p_\Delta(x)$ in (\ref{hatpx}) we obtain
$$ \hat m^{(\varepsilon)}(x)-\tilde m^{(\varepsilon)}(x)
= \tilde m^{(\varepsilon)}(x)\left(\frac{p_\Delta(x)}{\hat p_\Delta(x)} \mathbb{I}\{\hat p_\Delta(x)>0\}-1\right).$$
Together with \eqref{Rate-S} and $T_2=o_{\mathbb{P}}((\log n)^{-1})$, we conclude
\begin{eqnarray*} 
T_3&\le&\left( \mathbf{S}_\Delta + o_{\mathbb{P}}((\log n)^{-1})\right)\sup_{x\in \mathcal X}\left| \frac{p_\Delta(x)}{\hat p_\Delta(x)} \mathbb{I}\{\hat p_\Delta(x)>0\}-1\right|\\
&=& \mathcal{O}_{\mathbb{P}}(\sqrt{\log n})\sup_{x\in \mathcal X}\left| np_\Delta(x)\left(\frac{1}{n\hat p_\Delta(x)}-\frac{1}{np_\Delta(x)}\right) \mathbb{I}\{\hat p_\Delta(x)>0\}-\mathbb{I}\{\hat p_\Delta(x)=0\}\right|.
\end{eqnarray*}
To obtain the remainder rate $o_{\mathbb{P}}((\log n)^{-1})$, the second indicator can be written as $-(\mathbb{I}\{\exists j\in[n]:X_{j}\in A_{\Delta}(x)\}-1)$ and be handled as in (\ref{eq:prf histo approx err decomp}). For the other term note that for fixed $x$ the random variable $B_n=n\hat p_\Delta(x)$ is binomially distributed with parameters $n$ and  $p_n=p_\Delta(x)$, and we can apply Lemma~\ref{lem:moments binomial fractures}. 
Together with (\ref{eq:bounds p_x_0}), the union bound and Markov's inequality we obtain
\begin{align}\nonumber
    &\mathbb{P}\left( \sup_{x\in \mathcal X}\left| np_\Delta(x)\left(\frac{1}{n\hat p_\Delta(x)}-\frac{1}{np_\Delta(x)}\right)\right| \mathbb{I}\{\hat p_\Delta(x)>0\}>\frac{\kappa}{(\log n)^{3/2}}\right)
    \\
    &\leq \Delta \frac{(\log n)^{3q/2}}{\kappa^q} (nC_XC_p/\Delta)^q\mathcal{O}((n/\Delta)^{-3q/2})
    =\mathcal{O}\left(\Delta\frac{(\log n)^{3q/2}}{(n/\Delta)^{q/2}}\right) =o(1)
    \label{1/hatp}
 \end{align}
 which holds by assumption (\ref{eq:assum histo GP approx}) for $q\in 2\mathbb{N}$, $q\geq \nu$.  Hence, $T_3=o_{\mathbb{P}}((\log n)^{-1})$.

Altogether we conclude
\begin{equation}\label{repr}
\sqrt{n}  \sup_{x\in\mathcal X}\tau_\Delta(x)^{1/2}\vert\hat{m}(x)-m(x)\vert=\mathbf{S}_{\Delta}+R_n
\end{equation}
for a remainder term $R_n=o_{\mathbb{P}}((\log n)^{-1})$. Thus, the uniform coverage probability can be bounded from below by
\begin{align}
\mathbb{P} & \big(m(x)\in\mathcal{C}_{n}(x),\,\forall x\in\mathcal X\big)\nonumber \\
 & =\mathbb{P}\Big(\sqrt{n}  \sup_{x\in\mathcal X}\tau_\Delta(x)^{1/2}\vert\hat{m}(x)-m(x)\vert\leq c_{\Delta}(\beta)\Big)\nonumber \\
 & =\mathbb{P}\big(\mathbf{S}_{\Delta}+R_n\leq c_{\Delta}(\beta)\big)\label{Restterm} \\
 & \geq \mathbb{P}\big(\mathbf{S}_{\Delta}\leq c_{\Delta}(\beta)\big)- \mathbb{P}\big(c_{\Delta}(\beta)-\kappa_n<\mathbf{S}_{\Delta}\leq c_{\Delta}(\beta)\big)-\mathbb{P}\left(R_n>\kappa_n\right), \nonumber
 \\
& \geq 1-\beta- \mathbb{P}\left(|\mathbf{S}_{\Delta}-c_{\Delta}(\beta)|\leq \kappa_n\right)-\mathbb{P}\left(R_n>\kappa_n\right), \nonumber
\end{align}
where the last term converges to zero for $\kappa_n=(\log n)^{-1}$
and the probability in the middle term can be upper bounded by 
\begin{align}
\sup_{\xi\in\mathbb{R}}\mathbb{P}\left(\vert\mathbf{S}_{\Delta}-\xi\vert\leq\kappa_n\right) & \leq4\kappa_n\left(\mathbb{E}\left[\sup_{f\in\mathcal{F}_{\Delta}}\vert B_{\Delta}(f)\vert\right]+1\right)\nonumber \\
 & =4\kappa_n\left(\mathbb{E}\left[\max_{j=1,\ldots,\Delta}\vert Z_{j}\vert\right]+1\right)\nonumber \\
 & =\mathcal{O}(\kappa_n\sqrt{\log n}) =o(1)\nonumber
\end{align}
which follows from 
Corollary 2.1 by \citet{Chernozhukov2014a} and (\ref{Rate-S}).
Thus we obtain
\[
\liminf_{n\to\infty}\inf_{m\in\mathcal{H}^\alpha(C_{H})}\mathbb{P}\left(m(x)\in\mathcal{C}_{n}(x),\,\forall x\in\mathcal X\right)\geq1-\beta.\qedhere
\]
\end{proof}

\smallskip

\begin{proof}[Proof of Corollary \ref{limit-distribution}]
As in  (\ref{repr}) and (\ref{Restterm}) we obtain 
\begin{eqnarray*}
\mathbb{P} \big(m(x)\in\tilde{\mathcal{C}}_{n}(x),\,\forall x\in\mathcal{X}\big)&=&\mathbb{P}\big(\mathbf{S}_{\Delta}+R_n\leq \tilde{c}_{\Delta}(\beta)\big)\\
&=&\mathbb{P}\big(a_\Delta(\mathbf{S}_{\Delta}-b_\Delta)+a_\Delta R_n\leq -\log(-\log(1-\beta))\big)\\
&=&\mathbb{P}\big(a_\Delta(\mathbf{S}_{\Delta}-b_\Delta)\leq -\log(-\log(1-\beta))\big)+o(1).
\end{eqnarray*}
The last equality follows from Slutsky's lemma because $R_n=o_{\mathbb{P}}((\log n)^{-1})$, and from assumption (\ref{eq:assum histo GP approx})  it follows that $a_\Delta\lesssim(\log n)^{1/2}$. 
It remains to verify that the limit distribution of $a_\Delta(\mathbf{S}_{\Delta}-b_\Delta)\overset{d}{=} a_\Delta(\max_{j=1,\ldots,\Delta}\vert Z_{j}|-b_\Delta)$ is the Gumbel distribution with cdf $\exp(-e^{-x})$. Indeed, using the the asymptotic independence of the maximum and minimum \cite[Theorem 1.8.2]{leadbetter2012extremes} together with the extreme value limit of the Gaussian distribution \cite[Theorem 1.5.3]{leadbetter2012extremes}  we have
\begin{align*}
\P(a_{\Delta}(\max_j|Z_{j}|-b'_{\Delta}&)\le x)  =\P\Big(\max_j Z_{j}\le\frac{x}{a_{\Delta}}+b'_{\Delta},\min_j Z_{j}\ge-\frac{x}{a_{\Delta}}-b'_{\Delta}\Big)\\
 & =\P\Big(\max_j Z_{j}\le\frac{x}{a_{\Delta}}+b'_{\Delta}\Big)\P\Big(\min_j Z_{j}\ge-\frac{x}{a_{\Delta}}-b'_{\Delta}\Big)+o(1)\\
 & =\P\Big(a_{\Delta}(\max_j Z_{j}-b'_{\Delta})\le x\Big)\P\Big(-\max_j(-Z_{j})\ge-\frac{x}{a_{\Delta}}-b'_{\Delta}\Big)+o(1)\\
 & =\P\Big(a_{\Delta}(\max_j Z_{j}-b'_{\Delta})\le x\Big)^{2}+o(1)
  \to e^{-2e^{-x}}
\end{align*}
for $a_{\Delta}=\sqrt{2\log\Delta}$ and $b'_{\Delta}=\sqrt{2\log\Delta}-\frac{\log\log n+\log4\pi}{2\sqrt{2\log\Delta}}$.
Hence, $b_{\Delta}=b'_{\Delta}+a_{\Delta}^{-1}\log2$
yields
$\P(a_{\Delta}(\max|Z_{j}|-b_{\Delta})\le x)\to e^{-e^{-x}}$.
\end{proof}

\begin{proof}[Proof of Corollary \ref{cor-rate}]
To verify the uniform rate, we first note that assumption~\eqref{eq:assum histo approx} was only required to ensure that the approximation error $T_1$ from \eqref{approxterm} is of order $o_{\P}((\log n)^{-1})$. The estimates following \eqref{eq:prf histo approx err decomp} and \eqref{eq:bounds tau_x} show that the approximation error satisfies
\begin{align*}
  \E\big[\sup_{x\in\mathcal X}|\hat m^{(m)}(x)-m(x)|\big]
  \lesssim n^{-1/2}\Delta^{1/2}\E[T_1] 
  \lesssim \Delta^{-\alpha/p} +\Delta(1-c_Xc_p/\Delta)^n,
\end{align*}
where the second term is negligible if $\log(\Delta)=o(n/\Delta)$. The latter condition is implied by assumption~(\ref{eq:assum histo GP approx}). 
Together with $T_2+T_3=o_{\P}((\log n)^{-1})$, (\ref{Rate-S}) and (\ref{eq:bounds tau_x}), we conclude $\sup_{x\in\mathcal X}|\hat m(x)-m(x)|=\mathcal{O}_{\mathbb{P}}(\Delta^{-\alpha/p}+\sqrt{(\log n)\Delta/n})$. Choosing $\Delta=(n/\log n)^{1/(2\alpha/p +1)}$ balances both terms and yields the optimal rate of convergence.
\end{proof}

\smallskip

\begin{proof}[Proof of Corollary \ref{cor-estimated-px}]
As in the proof of Theorem \ref{thm:CB histo} we need to show $R_n=o_\mathbb{P}((\log n)^{-1})$ for the additional remainder term 
\begin{align*}
    R_n &= \bigg\vert\sqrt{n}  \sup_{x\in\mathcal X}\tau_\Delta(x)^{1/2}\vert\hat{m}(x)-m(x)\vert-\sqrt{n}  \sup_{x\in\mathcal X}\hat \tau_\Delta(x)^{1/2}\vert\hat{m}(x)-m(x)\vert\bigg\vert\\
    &\leq  \sqrt{n}  \sup_{x\in\mathcal X}\tau_\Delta(x)^{1/2}\vert\hat{m}(x)-m(x)\vert \sup_{x\in\mathcal X}\frac{|\hat \tau_\Delta(x)^{1/2}-\tau_\Delta(x)^{1/2}|}{\tau_\Delta(x)^{1/2}}\\
    &= \big(\mathbf{S}_\Delta + o_\mathbb{P}((\log n)^{-1})\big)\sup_{x\in\mathcal X}\frac{|\hat \tau_\Delta(x)-\tau_\Delta(x)|}{\tau_\Delta(x)^{1/2}(\tau_\Delta(x)^{1/2}+\hat \tau_\Delta(x)^{1/2})}\\
    &\leq \mathcal{O}_\mathbb{P}((\log n)^{1/2}) \sup_{x\in\mathcal X}\Big\vert \frac{\hat \tau_\Delta(x)}{\tau_\Delta(x)}- 1\Big\vert= o_\mathbb{P}((\log n)^{-1})
\end{align*}
by (\ref{repr}), (\ref{eq:dist S_delta}), (\ref{Rate-S}) and the assumption on $\hat\tau_\Delta$. 
\end{proof}

\subsection{Proof of Theorem \ref{thm:sup m tilde approx}}\label{subsec-auxiliary}

\begin{proof}[Proof of Theorem \ref{thm:sup m tilde approx}]
We will apply the result by \citet[Corollary 2.2]{Chernozhukov2014} formulated in Theorem \ref{thm:sup-approx Cherno} in the appendix to fix the notation. While
the function class in this result does not depend on $n$ and $\Delta$, the theorem can be applied to $\mathcal{F}_{\Delta}$ for every $n$
or $\Delta$, respectively, due to the universality of the occuring constants. As a consequence we get a sequence $\mathbf{S}_{\Delta}$ of approximating random variables as pointed out by \citet[Remark 2.1]{Chernozhukov2014}. We consider the function class $\mathcal{F}_{\Delta}$
from (\ref{eq:F delta}). To obtain the claimed result for the absolute value, we need to consider $\mathcal{F}_{\Delta}\cup-\mathcal{F}_{\Delta}$. Due to \citet[Corollary 2.1]{Chernozhukov2014}
we consider $\mathcal{F}_{\Delta}$ without loss of generality.

The functions $f_{x,\Delta}$ are measurable and thus the function
class is pointwise measurable. To see that $\mathcal{F}_{\Delta}$
is a VC-type class as defined in the appendix, we calculate a measurable envelope $F$:
\begin{align}
\sup_{x\in\mathcal X}\vert f_{x,\Delta}(t,\epsilon)\vert & =\sup_{x\in\mathcal X}\vert \epsilon\vert s_\Delta(x)^{-1/2}\mathbb{I}\{t\in A_{\Delta}(x)\}\nonumber \\
 & \leq\vert \epsilon\vert (c_{X}c_{\sigma^2}c_p)^{-1/2}\Delta^{1/2} =:F(t,\epsilon)\label{eq:envelope cherno}
\end{align}
applying $c_{\sigma^2}c_Xc_p\Delta^{-1}\leq s_\Delta(x)\leq C_{\sigma^2}C_XC_p\Delta^{-1}$ 
from \eqref{s_x}, \eqref{cX}, \eqref{csigma} and Assumption~\ref{ass:part}. The finite size $\vert\mathcal{F}_{\Delta}\vert=\Delta$
is an upper bound for any covering number of $\mathcal{F}_{\Delta}$. Hence,
\begin{equation}\label{vc-type}
    \sup_{Q\in\mathcal{Q}}N(\mathcal{F}_{\Delta},\Vert\cdot\Vert_{Q,2},\kappa\Vert F\Vert_{Q,2})\leq\Delta\leq\Delta/\kappa
\end{equation}
for all $\kappa\in(0,1]$, $\mathcal{F}_{\Delta}$ is indeed of VC-type and we can choose  $A=\Delta$ and $v=1$ in Theorem~\ref{thm:sup-approx Cherno}.

We proceed by verifying the conditions on the moments. Since
\[
\sup_{f\in\mathcal{F}_{\Delta}}P\vert f\vert^{2}=\sup_{x\in\mathcal X}  s_\Delta(x)^{-1}\mathbb{E}\left[\varepsilon_{1}^{2}\mathbb{I}\{X_{1}\in A_{\Delta}(x)\}\right]=1,
\]
the quantity $\tilde{\sigma}$ from Theorem \ref{thm:sup-approx Cherno} is equal to $1$.
Moreover,
\begin{align*}
\sup_{f\in\mathcal{F}_{\Delta}}P\vert f\vert^{3} & =\sup_{x\in\mathcal X} s_\Delta(x)^{-3/2}\mathbb{E}\left[\vert\varepsilon_{1}\mathbb{I}\{X_{1}\in A_{\Delta}(x)\}\vert^{3}\right]\\
 & \leq K(c_{\sigma^2}c_{X}c_p)^{-3/2}C_XC_p\Delta^{1/2}.
\end{align*}
For the envelope $F$ from (\ref{eq:envelope cherno}) we obtain
\[
\Vert F\Vert_{P,\nu}=\mathbb{E}\left[\vert F(X_{1},\varepsilon_{1})\vert^{\nu}\right]^{1/\nu}=(c_{X}c_{\sigma^2}c_p)^{-1/2}\Delta^{1/2}\mathbb{E}\left[\vert\varepsilon_{1}\vert^{\nu}\right]^{1/\nu}.
\]
We thus choose $b$ from Theorem \ref{thm:sup-approx Cherno} as
\[
b:=\Delta^{1/2}C_b,\quad C_{b}:=\max\left\{ K(c_Xc_{\sigma^2}C_p)^{-3/2}C_XC_p,\mathbb{E}\left[\vert\varepsilon_{1}\vert^{\nu}\right]^{1/\nu}(c_Xc_{\sigma^2}C_p)^{-1/2}\right\} .
\]
Since $b$ is of order $\Delta^{1/2}$, we have $b\geq1$ for $\Delta$ large enough. To apply Theorem \ref{thm:sup-approx Cherno} we obtain
\begin{align*}
\Gamma_{n} & =cv(\log n\vee\log(Ab/\tilde{\sigma}))
 =c(\log n\vee\log(\Delta^{3/2}C_{b}))\\
 & \lesssim \log n\vee p\log(\Delta) \lesssim\log n.
\end{align*}
Let $B_{\Delta}$ be the centered Gaussian process defined in the
claim. For $\gamma=(\log n)^{-1}$ Theorem \ref{thm:sup-approx Cherno}
yields that there exists a random variable
\[
\mathbf{S}_{\Delta}\overset{d}{=}\sup_{f\in\mathcal{F}_{\Delta}\cup-\mathcal{F}_{\Delta}}B_{\Delta}(f)=\sup_{f\in\mathcal{F}_{\Delta}}\vert B_{\Delta}(f)\vert=\sup_{x\in\mathcal X}\vert B_{\Delta}(f_{x,\Delta})\vert
\]
such that
\begin{align*}
\mathbb{P} & \Big(\Big\vert\sup_{f\in\mathcal{F}_{\Delta}}\vert\mathbb{G}_{n}f\vert-\mathbf{S}_{\Delta}\Big\vert>\frac{b\Gamma_{n}(\log n)^{1/2}}{n^{1/2-1/\nu}}+\frac{(b\log n)^{1/2}\Gamma_{n}^{3/4}}{n^{1/4}}+\frac{(b\Gamma_{n}^{2}\log n)^{1/3}}{n^{1/6}}\Big)\\
 & \lesssim (\log n)^{-1}+\frac{\log n}{n}.
\end{align*}
In view of \eqref{eq:gaussPr}, we conclude
\begin{align*}
\bigg\vert & \sqrt{n}\sup_{x\in\mathcal X}\vert\tau_\Delta(x)^{1/2}\tilde{m}^{(\varepsilon)}(x)\vert-\mathbf{S}_{\Delta}\bigg\vert\\
 & =\Big\vert\sup_{f\in\mathcal{F}_{\Delta}}\vert\mathbb{G}_{n}f\vert-\mathbf{S}_{\Delta}\Big\vert\\
 & =\mathcal{O}_{\mathbb{P}}\left(\frac{b\Gamma_{n}(\log n)^{1/2}}{n^{1/2-1/\nu}}+\frac{(b\log n)^{1/2}\Gamma_{n}^{3/4}}{n^{1/4}}+\frac{(b\Gamma_{n}^{2}\log n)^{1/3}}{n^{1/6}}\right)\\
 & =\mathcal{O}_{\mathbb{P}}\left(\frac{(\log n)^{3/2}\Delta^{1/2}}{n^{1/2-1/\nu}}+\frac{(\log n)^{5/4}\color{blue}\Delta^{1/4}}{n^{1/4}}+\frac{(\log n)\Delta^{1/6}}{n^{1/6}}\right)
\end{align*}
since $\Gamma_{n}\lesssim\log n$ and $b\lesssim \Delta^{1/2}$. Finally note, that the structure of the covariance \eqref{eq:covariance} implies that
\begin{equation}
\mathbf{S}_{\Delta}\overset{d}{=}\sup_{x\in\mathcal X}\vert B_{\Delta}(f_{x,\Delta})\vert\overset{d}{=}\max_{j=1,\ldots,\Delta}\vert Z_{j}\vert, \label{eq:dist S_delta}
\end{equation}
for independent standard normally distributed $Z_{1},Z_2,\dots$.
\end{proof}

\subsection{Moment bounds for binomial distributions}\label{sec:binom}

We need a uniform bound on centered moments of binomial random variables. To this end, we use \citet[Theorem 4]{Skorski2025}, which gives for any random variable $B\sim\text{Bin}(n,p)$ with $n\in\mathbb N$, $p\in(0,1)$ and for any $q>1$ that
\[
\mathbb{E}\big[(B-np)^{q}\big]^{1/q}=C(n,p,q)\max\big\{k^{1-k/q}(np(1-p))^{k/q}:k=1,\ldots,\lfloor q/2\rfloor\big\},
\]
where $C(n,p,q)$ is uniformly bounded by $(3e)^{-1}\leq C(n,p,q)\leq(5/2)^{1/5}e^{1/2}$. If $np>1$, we obtain
\[
\max\{k^{1-k/q}(n\sigma^{2})^{k/q}:k=1,\ldots,\lfloor q/2\rfloor\}\leq(np)^{\lfloor q/2\rfloor/q}\max\{k^{1-k/q}:k=1,\ldots,\lfloor q/2\rfloor\}.
\]
Hence, for fixed $q>1$ we have
\begin{align}\label{eq:BinomBound}
  \mathbb{E}[(B-np)^{q}]\lesssim(np)^{q/2}\text{ uniformly in }n\text{ and }p.
\end{align}
\begin{lem}
\label{lem:moments binomial fractures}Let $B\sim\text{Bin}(n,p)$
such that $p^{-1}\lesssim n$. For any fixed integer $q>1$ it
holds that
\begin{align*}
\mathbb{E}\left[\mathbb{I}\{B>0\}\left(\frac{1}{B}-\frac{1}{np}\right)^{q}\right] & \lesssim(np)^{-3q/2}.
\end{align*}
\end{lem}

\begin{proof}
Using simple calculations one gets
\begin{equation}
\frac{1}{B}-\frac{1}{np}=\frac{np-B}{(np)^{2}}+\frac{(np-B)^{2}}{(np)^{3}}+\frac{(np-B)^{3}}{B(np)^{3}}.\label{eq:binom denom decomp}
\end{equation}
Hence we get
\begin{align*}
\mathbb{E} & \left[\mathbb{I}\{B>0\}\left(\frac{1}{B}-\frac{1}{np}\right)^{q}\right]\\
 & \lesssim\mathbb{E}\left[\left(\frac{np-B}{(np)^{2}}\right)^{q}\right]+\mathbb{E}\left[\left(\frac{(np-B)^{2}}{(np)^{3}}\right)^{q}\right]+\mathbb{E}\left[\mathbb{I}\{B>0\}\left(\frac{(np-B)^{3}}{B(np)^{3}}\right)^{q}\right]\\
 & \leq(np)^{-2q}\mathbb{E}\left[\left(np-B\right)^{q}\right]+(np)^{-3q}\mathbb{E}\left[(np-B)^{2q}\right]+(np)^{-3q}\mathbb{E}\left[(np-B)^{3q}\right].
\end{align*}
We apply \eqref{eq:BinomBound} in the case where $p^{-1}=\mathcal{O}(n)$.
We obtain
\begin{align*}
\mathbb{E} & \left[\mathbb{I}\{B>0\}\left(\frac{1}{B}-\frac{1}{np}\right)^{q}\right]\\
 & \lesssim(np)^{-2q}\mathbb{E}\left[(np-B)^{q}\right]+(np)^{-3q}\mathbb{E}\left[(np-B)^{2q}\right]+(np)^{-3q}\mathbb{E}\left[(np-B)^{3q}\right]\\
 & \lesssim(np)^{-3q/2}+(np)^{-2q}+(np)^{-3q/2}
 \lesssim(np)^{-3q/2},
\end{align*}
which yields the claim.
\end{proof}

\begin{lem}\label{lem:pHat}
  Let Assumption~\ref{ass:part} and \eqref{eq:assum histo GP approx} be satisfied. Then we have
  \[
  \sup_{x\in\mathcal X}\Big\vert \frac{\hat p_\Delta(x)}{p_\Delta(x)}- 1\Big\vert=o_\mathbb{P}\big((\log n)^{-3/2})\big).
  \]
\end{lem}
\begin{proof}
Note that
\begin{align*}
    \sup_{x\in\mathcal X}\Big\vert \frac{\hat p_\Delta(x)}{p_\Delta(x)}- 1\Big\vert
    &= \sup_{x\in\mathcal X}\frac{1}{n p_\Delta(x)}\Big|\sum_{i=1}^n \mathbb{I}\{X_i\in A_\Delta(x)\}-np_\Delta(x)\Big|.
\end{align*}
Since $x\mapsto\hat p_\Delta(x)$ is piecewise constant with at most $\Delta$ different values, the supremum can be handled with a union bound. Moreover, we can exploit that $\sum_{i=1}^n \mathbb{I}\{X_i\in A_\Delta(x)\}$ is $Bin(n,p_\Delta(x))$-distributed. The binomial moment result \eqref{eq:BinomBound}  and (\ref{eq:bounds p_x_0}) yield for any $q\in\mathbb{N}$ that
\begin{align*}
&\mathbb{P}\left(  \sup_{x\in\mathcal X}\Big\vert \frac{\hat p_\Delta(x)}{p_\Delta(x)}- 1\Big\vert
\geq\frac{\kappa}{(\log n)^{3/2}}\right)\\
 &\qquad \leq   \frac\Delta{ ( nc_Xc_p/\Delta )^{q} }
 \mathbb{E}\Big[\Big|
 \sum_{i=1}^n \mathbb{I}\{X_i\in A_\Delta(x)\}-np_\Delta(x)\Big|^{q}\Big]
 \frac{(\log n)^{3q/2}}{\kappa^{q}}
 \\
 &\qquad \lesssim\frac{\Delta(\log n)^{3q/2}}{\kappa^q(n/\Delta)^{q/2}} =o(1)
 \end{align*}
 which holds by assumption (\ref{eq:assum histo GP approx}) for large enough $q$.
\end{proof}

\appendix

\section{Extensions}\label{sec:extension}
For the general partition in Assumption \ref{ass:part} we chose disjoint cells, but intersecting  cells are also possible. For this we  fix a mapping $x\mapsto A_\Delta(x)$, where all $(A_\Delta(x))_{x\in\mathcal X}$ satisfy (\ref{ass:cell}), and we assume that $\{ A_\Delta(x)\mid x\in\mathcal{X}\}$ is a subset of some VC-class of sets, not depending on $\Delta$. For instance the cells $A_\Delta(x)$ can be chosen as balls (or hypercubes) with middle point $x$ and radius (side length) depending on $\Delta$. We use the same notation for the estimator $\hat m$ in (\ref{hat-m}) and the assumptions in Theorem \ref{thm:CB histo}. However, $\Delta$ is not the number of cells anymore, and we do not assume that there are only finitely many cells for each $\Delta$. Thus the parts of the proof applying the union bound and the finite cardinality of the function class $\mathcal{F}_\Delta$ have to be modified. The general proof structure is still the same as in Sections \ref{sec:Proof-strategy} and \ref{sec:Proofs}. 
\\
One starts with the same expansion with terms $T_1$, $T_2$ and $T_3$ from (\ref{approxterm}), (\ref{approx-leadingterm}) and (\ref{msigmaremainder}), but with an adjusted $\mathbf{S}_{\Delta}$ discussed below. To handle $T_1$, (\ref{eq:prf histo approx err part 1 bound}) works out the same way, but for  (\ref{eq:prf histo approx err decomp}) and showing negligibility of $T_3$ a different argument is needed. 
Let $P$ denote the distribution of $X_1$ and $P_n$ the empirical distribution based on $X_1,\dots,X_n$. Then one can write 
$$\sup_{x\in\mathcal{X}}|\hat p_\Delta(x)- p_\Delta(x)|=\sup_{g\in\mathcal{G}_n}|P_n g-Pg|$$ 
for the function class $\mathcal{G}_n=\{\mathbb{I}_{A_\Delta(x)}\mid x\in \mathcal{X}\}$ (with $\Delta=\Delta_n$), which is a subset of a VC-class not depending on $n$. We have  $|g|\leq 1$ and $(Pg^2)^{1/2}=(p_\Delta(x))^{1/2}\leq (C_XC_p)^{1/2} \Delta^{-1/2} $ for all $g\in\mathcal{G}_n$ by (\ref{eq:bounds p_x_0}) and (\ref{ass:cell}). Applying Theorem 37 in \citet[p.\ 34]{Pollard}, one obtains
$$\sup_{x\in\mathcal{X}}|\hat p_\Delta(x)- p_\Delta(x)|=o_{\P}(\alpha_n\Delta^{-1})$$
for each positive non-increasing sequence $\alpha_n$ with $(\Delta \log n)/(n\alpha_n^2 )=o(1)$. In view of (\ref{eq:assum histo GP approx}) we can choose $\alpha_n=(\log n)^{-2}$. With this one also obtains a similar result as (\ref{1/hatp}) because
\begin{align*}
   & \sup_{x\in \mathcal X}\left| p_\Delta(x)\left(\frac{1}{\hat p_\Delta(x)}-\frac{1}{p_\Delta(x)}\right)\right| \mathbb{I}\{\hat p_\Delta(x)>0\}\\
    \leq& 
\frac{\sup_{x\in\mathcal{X}}|\hat p_\Delta(x)- p_\Delta(x)|}{\inf_{x\in\mathcal{X}}p_\Delta(x)}\sup_{x\in\mathcal{X}}\frac{p_\Delta(x)}{p_\Delta(x)-\sup_{t\in\mathcal{X}}|\hat p_\Delta(t)- p_\Delta(t)|}=o_{\P}((\log n)^{-2}).
\end{align*}
Further for the term $\mathbb{I}\{\hat p_\Delta(x)=0\}$ in the upper bound of $T_3$ note that  for $\eta>0$
\begin{align*}
\mathbb{P}\left((\log n)^{3/2}\sup_{x\in\mathcal{X}}\mathbb{I}\{\hat p_\Delta(x)=0\}>\eta\right)
\leq& 
\mathbb{P}\left(\exists x\in\mathcal{X}:\hat p_\Delta(x)=0\right)\\
\leq& 
\mathbb{P}\left(\sup_{x\in\mathcal{X}}|\hat p_\Delta(x)-p_\Delta(x)|\geq c_Xc_p\Delta^{-1}\right)=o(1),
\end{align*}
and to handle term (\ref{eq:prf histo approx err decomp}) for $T_1$ works in the same way.
Thus $T_1$ is negligible. For negligibility of $T_3$ one also needs the rate $\mathbf{S}_\Delta=\mathcal{O}_{\mathbb{P}}(\sqrt{\log n})$, and arguments for this are based on the cell structures. We consider this in the examples below. The dominating term $T_2$ can be treated as in Theorem \ref{thm:sup m tilde approx}. For this let $\mathcal{A}_\Delta(x)$ be the VC-class containing all sets $A_\Delta(x)$. Then 
$\mathcal{F}_\Delta$ with the analogous definition as in (\ref{eq:F delta}) and envelope $F$ from (\ref{eq:envelope cherno}) is a subset of the VC-type class
$$\{(t,\epsilon)\mapsto |\epsilon|c\Delta^{1/2}\mathbb{I}_A\mid A\in \mathcal{A}_\Delta(x), c\in[0,(c_Xc_{\sigma^2}c_p)^{-1/2}],x\in\mathcal X\}$$ such that the covering number is bounded by
$$\sup_{Q\in\mathcal{Q}}N(\mathcal{F}_{\Delta},\Vert\cdot\Vert_{Q,2},\kappa\Vert F\Vert_{Q,2})\leq C\kappa^{-V}$$
for some $C$ and $V$ not depending on $\Delta$. So we even get a better upper bound than in (\ref{vc-type}) 
and the remaining part of the proof of Theorem \ref{thm:sup m tilde approx} can be done analogously. 
The approximation  has the structure
\begin{equation*}
\mathbf{S}_{\Delta}\overset{d}{=}\sup_{x\in\mathcal X}\vert B_{\Delta}(f_{x,\Delta})|,
\end{equation*}
where the centered Gaussian process $B_\Delta$ has the covariance formula
$$\cov( B_{\Delta}(f_{x_1,\Delta}),  B_{\Delta}(f_{x_2,\Delta}))=\frac{\int_{A_\Delta(x_1)\cap A_\Delta(x_2)} \sigma^2(t)f_X(t)\,dt}{(s_\Delta(x_1)s_\Delta(x_2))^{1/2}},$$
and the same distribution as the process
\begin{equation}\label{eq:W}
\frac{1}{(s_\Delta(x))^{1/2}}\int \sigma(t)(f_X(t))^{1/2}\mathbb{I}_{A_\Delta(x)}(t)\, dW(t)
\end{equation}
with a $p$-parameter Brownian sheet $W$. Due to the overlapping cells, we don't have an diagonal covariance structure anymore and $\mathbf S_\Delta$ cannot be represented by a finite number of Gaussian random variables anymore.
\\
\begin{example}
    Consider the homoscedastic case with $X\sim U[0,1]^p$. Then the covariance structure simplifies to 
    $$ \cov( B_{\Delta}(f_{x_1,\Delta}),  B_{\Delta}(f_{x_2,\Delta}))=\frac{\vol (A_\Delta(x_1)\cap A_\Delta(x_2))}{(\vol (A_\Delta(x_1))\vol (A_\Delta(x_2)))^{1/2}},  $$ 
    and quantiles of $\mathbf{S}_{\Delta}$ can be directly simulated to obtain a confidence band. For the proof discussed above one needs the convergence rate of $\mathbf{S}_\Delta$. Let us consider the case where $A_\Delta(x)$ is a ball around $x$ (intersected with $[0,1]^p$) with radius $r_\Delta(x)$, which may depend on $x$ with the assumption
    $$ c_r\Delta^{-1/p}\leq r_\Delta(\cdot)\leq C_r\Delta^{-1/p}$$
    for some constants $C_r\geq c_r>0$. Define the distance between $x_1,x_2\in[0,1]^p$ as
    $$d(x_1,x_2)=\big(\var(B_{\Delta}(f_{x_1,\Delta})-B_{\Delta}(f_{x_2,\Delta}))\big)^{1/2}=\left(1-\frac{\vol (A_\Delta(x_1)\cap A_\Delta(x_2))}{(\vol (A_\Delta(x_1))\vol (A_\Delta(x_2)))^{1/2}}\right)^{1/2},$$
    and assume it can be upper bounded by 
    $$d(x_1,x_2)\leq \left(c \|x_1-x_2\|\Delta^{1/p}\right)^{1/2}$$
    for some constant $c$. For instance for dimension $p=3$ one can apply a sphere-sphere intersection formula and assume Lipschitz continuity of the map $x\mapsto r_\Delta(x)$. 
    Then one obtains an upper bound for the covering number $N(\epsilon,\mathcal{F}_\Delta,d)\lesssim (c\Delta^{1/p}/\epsilon^2)^p$. By \citet[Theorem 3.18]{Massart} we obtain
    $$\E [\mathbf{S}_\Delta]  \lesssim\int_0^1 \sqrt{\log((c\Delta^{1/p}/\epsilon^2)^p)}\, d\epsilon=\mathcal{O}(\sqrt{\log\Delta})$$
    and thus $\mathbf{S}_\Delta=\mathcal{O}_{\mathbb{P}}(\sqrt{\log n})$. 
\end{example}
\smallskip
\begin{example}
     Consider the special case where $t\in A_\Delta(x)$ is equivalent to $(x-t)/\delta\in A$ for one fixed set $A$ such that one can write $\mathbb{I}\{t\in A_\Delta(x)\}=K(\frac{x-t}{\delta})$ with  $K=\mathbb{I}_A$ and for example $A=[-1,1]^p$. For the assumptions on $\Delta$ as in Example \ref{example} let $\Delta=\delta^{-p}$. The estimator $\hat m$ in (\ref{hat-m}) is then a Nadaraya-Watson estimator with kernel $K$ and bandwidth $\delta$. 
As in the last steps of classical confidence band proof structures, see  \cite{Chao2017}, one can now replace $\sigma(t)(f_X(t))^{1/2}$ inside the integral (\ref{eq:W}) by $\sigma(x)(f_X(x))^{1/2}$ and use a factor function to obtain the simple process
$$Z_\delta(x)=\frac{1}{\sqrt{\delta^p}}\int K\left(\frac{x-t}{\delta}\right)\, dW(t).$$
For the replacement some assumptions on the functions $\sigma$ and $f_X$ are needed. For the proof structure considered above note that instead of the rate of $\mathbf{S}_\Delta$ one can use the rate $\sup_x |Z_\delta(x)|=\mathcal{O}_{\P}(\sqrt{\log n})$, and one obtains the limit distribution with the classical \cite{Rosenblatt} result. 
We are not aware of literature about asymptotic confidence bands based on Nadaraya-Watson estimators in the multidimensional covariate case for regression functions with  H\"older regularity $\alpha\in(0,1]$. This additionally demonstrates the usefulness of the new proof structure.
\end{example}

\section{Approximation of suprema of empirical processes}\label{appendix}

Denote by $P_n$ the empirical distribution of independent random variables $Z_1,\dots,Z_n$ with distribution $P$. Then we use the empirical process notation 
$\mathbb{G}_{n}f =n^{1/2}(P_nf-Pf)$ indexed in $f\in\mathcal{F}$, where we assume that the function class $\mathcal{F}$ is of VC-type with envelope $F$. This means that there
exist constants $A,v>0$ such that for the covering number 
\[
\sup_{Q\in\mathcal{Q}}N(\mathcal{F},\Vert \cdot\Vert_{Q,2},\varepsilon\Vert F\Vert_{Q,2})\leq(A/\varepsilon)^{v}\qquad\text{for all}\;\varepsilon\in(0,1]
\]
is valid. Here we use the notation $\Vert f\Vert_{Q,2}=(Qf^2)^{1/2}$ and $\mathcal Q$ denotes the set of all finitely discrete probability measures.
The next theorem is Corollary 2.2 by \citet{Chernozhukov2014}. 

\begin{thm}
\label{thm:sup-approx Cherno}Suppose that $\mathcal{F}$ is a pointwise
measurable function class with envelope $F$ that is VC-type for constants
$A\geq e,v\geq1$. Let $P$ denote the distribution of $Z_{1}$  and let $G_{P}$ be a centered Gaussian
process indexed in $\mathcal{F}$ with
\[
\mathbb{E}[G_{P}(f)G_{P}(g)]=P(fg)=\mathbb{E}[f(Z_{1})g(Z_{1})].
\]
Suppose also that for some $b\geq\tilde{\sigma}>0,$ and $\nu\in[4,\infty]$,
we have $\sup_{f\in\mathcal{F}}P\vert f\vert^{k}\leq\tilde{\sigma}^{2}b^{k-2}$
for $k=2,3$ and $\Vert F\Vert_{P,\nu}\leq b$. Then for every $\gamma\in(0,1)$ there exists a random variable $\mathbf{S}\overset{d}{=}\sup_{f\in\mathcal{F}}G_{P}f$
such that
\[
\mathbb{P}\Big(\Big\vert\sup_{f\in\mathcal{F}}\mathbb{G}_{n}f-\mathbf{S}\Big\vert>\frac{b\Gamma_{n}}{\gamma^{1/2}n^{1/2-1/\nu}}+\frac{(b\tilde{\sigma})^{1/2}\Gamma_{n}^{3/4}}{\gamma^{1/2}n^{1/4}}+\frac{(b\tilde{\sigma}^{2}\Gamma_{n}^{2})^{1/3}}{\gamma^{1/3}n^{1/6}}\Big)\leq C\Big(\gamma+\frac{\log n}{n}\Big)
\]
where $\Gamma_{n}=cv(\log n\vee\log(Ab/\tilde{\sigma}))$, and $c,C>0$
are constants that depend only on $\nu$ ($1/\nu$ is interpreted
as $0$ when $\nu=\infty$).
\end{thm}

\bigskip

\bigskip

\noindent{\bf  Acknowledgments. }
The authors are grateful to the Editor-in-Chief, Dietrich von Rosen, an Associate Editor and the reviewers for their valuable comments, which led to an improved manuscript. Moreover, we thank Max Farrell for his insightful remarks.

\bibliographystyle{apalike}
\phantomsection\addcontentsline{toc}{section}{\refname}\bibliography{rf_lit}

\end{document}

%% file: img/quantiles2.tex
\begin{tikzpicture}[x=.5pt,y=.5pt]
\definecolor{fillColor}{RGB}{255,255,255}
\path[use as bounding box,fill=fillColor,fill opacity=0.00] (0,60) rectangle (813.79,500);
\begin{scope}
\path[clip] ( 49.20, 61.20) rectangle (788.59,474.00);
\definecolor{drawColor}{RGB}{0,0,255}

\path[draw=drawColor,line width= 0.8pt,line join=round,line cap=round] ( 76.58,419.71) --
	(138.82,203.09) --
	(201.06,143.24) --
	(263.30,117.93) --
	(325.54,104.79) --
	(387.78, 97.07) --
	(450.01, 92.14) --
	(512.25, 88.80) --
	(574.49, 86.42) --
	(636.73, 84.67) --
	(698.97, 83.35) --
	(761.21, 82.33);
\end{scope}
\begin{scope}
\path[clip] (  0.00,  0.00) rectangle (813.79,523.20);
\definecolor{drawColor}{RGB}{0,0,0}

\path[draw=drawColor,line width= 0.4pt,line join=round,line cap=round] ( 49.20, 61.20) --
	(788.59, 61.20) --
	(788.59,474.00) --
	( 49.20,474.00) --
	cycle;
\end{scope}
\begin{scope}
\path[clip] (  0.00,  0.00) rectangle (813.79,523.20);
\definecolor{drawColor}{RGB}{0,0,0}

\node[text=drawColor,anchor=base,inner sep=0pt, outer sep=0pt, scale= 0.8] at (418.90, 15.60) {$c_\Delta(0.05)$};

\end{scope}
\begin{scope}
\path[clip] ( 49.20, 61.20) rectangle (788.59,474.00);
\definecolor{drawColor}{RGB}{255,0,0}

\path[draw=drawColor,line width= 0.8pt,dash pattern=on 4pt off 4pt ,line join=round,line cap=round] ( 76.58,199.08) --
	(138.82,112.93) --
	(201.06, 92.97) --
	(263.30, 85.58) --
	(325.54, 82.12) --
	(387.78, 80.27) --
	(450.01, 79.17) --
	(512.25, 78.47) --
	(574.49, 78.00) --
	(636.73, 77.68) --
	(698.97, 77.44) --
	(761.21, 77.26);
\end{scope}
\begin{scope}
\path[clip] (  0.00,  0.00) rectangle (813.79,523.20);
\definecolor{drawColor}{RGB}{0,0,0}

\path[draw=drawColor,line width= 0.4pt,line join=round,line cap=round] ( 49.20, 61.20) -- (761.21, 61.20);

\path[draw=drawColor,line width= 0.4pt,line join=round,line cap=round] (138.82, 61.20) -- (138.82, 55.20);

\path[draw=drawColor,line width= 0.4pt,line join=round,line cap=round] (263.30, 61.20) -- (263.30, 55.20);

\path[draw=drawColor,line width= 0.4pt,line join=round,line cap=round] (387.78, 61.20) -- (387.78, 55.20);

\path[draw=drawColor,line width= 0.4pt,line join=round,line cap=round] (512.25, 61.20) -- (512.25, 55.20);

\path[draw=drawColor,line width= 0.4pt,line join=round,line cap=round] (636.73, 61.20) -- (636.73, 55.20);

\path[draw=drawColor,line width= 0.4pt,line join=round,line cap=round] (761.21, 61.20) -- (761.21, 55.20);

\node[text=drawColor,anchor=base,inner sep=0pt, outer sep=0pt, scale= 0.8] at (138.82, 39.60) {$10^2$};

\node[text=drawColor,anchor=base,inner sep=0pt, outer sep=0pt, scale= 0.8] at (263.30, 39.60) {$10^4$};

\node[text=drawColor,anchor=base,inner sep=0pt, outer sep=0pt, scale= 0.8] at (387.78, 39.60) {$10^6$};

\node[text=drawColor,anchor=base,inner sep=0pt, outer sep=0pt, scale= 0.8] at (512.25, 39.60) {$10^8$};

\node[text=drawColor,anchor=base,inner sep=0pt, outer sep=0pt, scale= 0.8] at (636.73, 39.60) {$10^{10}$};

\node[text=drawColor,anchor=base,inner sep=0pt, outer sep=0pt, scale= 0.8] at (761.21, 39.60) {$10^{12}$};

\path[draw=drawColor,line width= 0.4pt,line join=round,line cap=round] ( 49.20, 76.49) -- ( 49.20,458.72);

\path[draw=drawColor,line width= 0.4pt,line join=round,line cap=round] ( 49.20, 76.49) -- ( 43.20, 76.49);

\path[draw=drawColor,line width= 0.4pt,line join=round,line cap=round] ( 49.20,203.90) -- ( 43.20,203.90);

\path[draw=drawColor,line width= 0.4pt,line join=round,line cap=round] ( 49.20,331.31) -- ( 43.20,331.31);

\path[draw=drawColor,line width= 0.4pt,line join=round,line cap=round] ( 49.20,458.72) -- ( 43.20,458.72);

\node[text=drawColor,rotate= 90.00,anchor=base,inner sep=0pt, outer sep=0pt, scale= 0.8] at ( 34.80, 76.49) {0.0};

\node[text=drawColor,rotate= 90.00,anchor=base,inner sep=0pt, outer sep=0pt, scale= 0.8] at ( 34.80,203.90) {0.1};

\node[text=drawColor,rotate= 90.00,anchor=base,inner sep=0pt, outer sep=0pt, scale= 0.8] at ( 34.80,331.31) {0.2};

\node[text=drawColor,rotate= 90.00,anchor=base,inner sep=0pt, outer sep=0pt, scale= 0.8] at ( 34.80,458.72) {0.3};
\end{scope}
\begin{scope}
\path[clip] ( 49.20, 61.20) rectangle (788.59,474.00);
\definecolor{drawColor}{RGB}{211,211,211}

\path[draw=drawColor,line width= 0.4pt,dash pattern=on 2pt off 2pt on 6pt off 2pt ,line join=round,line cap=round] ( 49.20, 76.49) -- (788.59, 76.49);

\path[draw=drawColor,line width= 0.4pt,dash pattern=on 2pt off 2pt on 6pt off 2pt ,line join=round,line cap=round] ( 49.20,140.19) -- (788.59,140.19);

\path[draw=drawColor,line width= 0.4pt,dash pattern=on 2pt off 2pt on 6pt off 2pt ,line join=round,line cap=round] ( 49.20,203.90) -- (788.59,203.90);

\path[draw=drawColor,line width= 0.4pt,dash pattern=on 2pt off 2pt on 6pt off 2pt ,line join=round,line cap=round] ( 49.20,267.60) -- (788.59,267.60);

\path[draw=drawColor,line width= 0.4pt,dash pattern=on 2pt off 2pt on 6pt off 2pt ,line join=round,line cap=round] ( 49.20,331.31) -- (788.59,331.31);

\path[draw=drawColor,line width= 0.4pt,dash pattern=on 2pt off 2pt on 6pt off 2pt ,line join=round,line cap=round] ( 49.20,395.01) -- (788.59,395.01);

\path[draw=drawColor,line width= 0.4pt,dash pattern=on 2pt off 2pt on 6pt off 2pt ,line join=round,line cap=round] ( 49.20,458.72) -- (788.59,458.72);
\definecolor{drawColor}{RGB}{0,0,255}

\path[draw=drawColor,line width= 0.8pt,line join=round,line cap=round] (565,445) -- (590.82,445);
\definecolor{drawColor}{RGB}{255,0,0}

\path[draw=drawColor,line width= 0.8pt,dash pattern=on 4pt off 4pt ,line join=round,line cap=round] (565,405) -- (590.82,405);
\definecolor{drawColor}{RGB}{0,0,0}

\node[text=drawColor,anchor=base west,inner sep=0pt, outer sep=0pt, scale= 0.8] at (600,440) {$\tilde c_\Delta(0.05)-c_\Delta(0.05)$};

\node[text=drawColor,anchor=base west,inner sep=0pt, outer sep=0pt, scale= 0.8] at (600,400) {$\displaystyle\frac{\tilde c_\Delta(0.05)-c_\Delta(0.05)}{c_\Delta(0.05)}$};
\end{scope}
\end{tikzpicture}